\documentclass[11pt,a4paper]{amsart}
\numberwithin{equation}{section}
\usepackage[margin=3.2cm]{geometry}
\usepackage{url}

\theoremstyle{plain}
\newtheorem{Th}{Theorem}[section]
\newtheorem{Lemma}[Th]{Lemma}

\newtheorem{Prop}[Th]{Proposition}
\newtheorem{innercustomthm}{Theorem}
\newenvironment{customthm}[1]
  {\renewcommand\theinnercustomthm{#1}\innercustomthm}
  {\endinnercustomthm}

\theoremstyle{definition}
\newtheorem{Def}[Th]{Definition}

\newtheorem{?}[Th]{Problem}

\theoremstyle{remark}
\newtheorem*{remark}{Remark}

\usepackage{amstext,amsfonts,amssymb,amscd,amsbsy,amsmath,tikz,mathdots}
\usepackage{enumerate}
\usepackage{graphicx}
\usepackage{subcaption}
\usepackage{mathtools}
\def\R{\mathbb R}
\def\N{\mathbb N}

\def\IDiam{\operatorname{IDiam}}

\def\bdiam{\operatorname{bdiam}}

\DeclarePairedDelimiter\parentheses{\lparen}{\rparen}
\newcommand{\interior}[1]{\operatorname{int} \parentheses*{#1}}

\begin{document}

\title[ITFF Equivalent to IDiam]{Intrinsic Tame Filling Functions are Equivalent to Intrinsic Diameter Functions}

\author[A. Hayes]{Andrew Hayes}

\maketitle

\begin{abstract}
Intrinsic tame filling functions are quasi-isometry invariants that are refinements of the intrinsic diameter function of a group.  The main purpose of this paper is to show that every finite presentation of a group has an intrinsic tame filling function that is equivalent to its intrinsic diameter function.  We also introduce some alternative filling functions---based on concepts similar to those used to define intrinsic tame filling functions---that are potential proper refinements of the intrinsic diameter function.

\bigskip

\noindent 2020 Mathematics Subject Classification:  20F65; 20F05, 20F06.

\bigskip

\noindent Key words: Filling functions, isodiametric functions, tame filling functions.
\end{abstract}

\section{Introduction}

Filling functions are quasi-isometry invariants for groups that describe the growth of combinatorial and geometric properties of van Kampen diagrams.  Because they are defined using van Kampen diagrams, any given filling function is only defined for a specific presentation of the group.  However, under an appropriate equivalence relation (see Definition~\ref{equivalent definition}), the equivalence class of a filling function for a presentation is a quasi-isometry invariant, and therefore a group invariant.  These invariants often tell us something about the complexity of the word problem for the group.  For an introduction to filling functions, see \cite[Chapter II]{FFs} by Riley.

One well-studied filling function is the intrinsic diameter (or isodiametric) function, which measures the diameter of van Kampen diagrams.  The intrinsic diameter function is one of several filling functions that is recursive if and only if the word problem of the group is solvable (see \cite[Th. 2.1]{lineardiam}).  As a result, there are groups whose intrinsic diameter functions grow faster than any recursive function.  It is also known that the collection of intrinsic diameter functions includes a variety of different kinds of growth.  For example, \cite[Cor. 1.4]{functions_that_can_be_IDiam} shows that this collection includes a wide variety of power functions.

The intrinsic diameter function is intrinsic in the sense that the diameter is measured using the path metric on the 1-skeleton of the van Kampen diagram itself.  Replacing this metric with the path metric of the Cayley graph gives an extrinsic diameter function, which is bounded above by the intrinsic diameter function.  Extrinsic diameter functions were introduced by Bridson and Riley in \cite{IDvsED}, where they also showed that there can exist arbitrarily large polynomial gaps between intrinsic and extrinsic diameter functions: for any $\alpha > 0$, there is a finite presentation with intrinsic diameter function bounded below by a polynomial $p$ and extrinsic diameter function bounded above by a polynomial $q$, with $\text{deg}(p) - \text{deg}(q) = \alpha$.

However, many groups are known to have at most linear intrinsic, and therefore also at most linear extrinsic, diameter functions, meaning that the equivalence class of the diameter functions cannot distinguish between these groups.  Among them are all combable groups and almost convex groups \cite[Prop. 3.4, Cor. 4.3] {lineardiam}, and therefore all automatic groups \cite[Lemma 2.3.2]{wordproc}.  These classes include word hyperbolic groups \cite[Th. 3.4.5]{wordproc}, groups acting geometrically on CAT(0) cube complexes \cite{CAT0}, and Euclidean groups \cite[Cor. 4.1.6]{wordproc}, \cite[Th. 2.5]{AC}; Coxeter groups \cite{Coxeter, CoxeterAC}; Artin groups of finite type \cite{Artinfin} and sufficiently large type \cite{Artinsuff}; and mapping class groups of surfaces of finite type \cite{MCGs}, among others.  The fundamental groups of all compact 3-manifolds also have at most linear diameter functions \cite[Th. 3.5]{lineardiam}.  The class of groups with at most linear diameter functions is also closed under graph products \cite{IDiamgraphprod}.

In part because such a large class of groups have linear diameter functions, Brittenham and Hermiller defined new quasi-isometry invariants in \cite{TFFs} called intrinsic and extrinsic tame filling functions.  The following is an informal description of tame filling functions; see Section~2 for formal definitions.  Tame filling functions are defined using a \emph{filling} of a presentation---a choice of a van Kampen diagram for each word over the generators that is equal to the identity in the group---and a \emph{1-combing} for each of these van Kampen diagrams, which amounts to a choice of continuously-varying paths from the basepoint of a van Kampen diagram to its boundary.  Given a point on one of these paths a distance~$n$ from the basepoint of the diagram, a tame filling function $f$ mandates that no prior point on the path can be a distance more than $f(n)$ from the basepoint; a 1-combing satisfying this property is called \emph{$f$-tame}.  Informally, a van Kampen diagram can be considered a terrain where elevation is represented by distance from the basepoint.  A group with a slow-growing tame filling function is intended to have van Kampen diagrams that are relatively smooth inclines, extending steadily outward from the basepoint to the boundary.  A group with a fast-growing tame filling function, on the other hand, may require van Kampen diagrams with tall mountains and deep valleys.

Brittenham and Hermiller showed that any intrinsic or extrinsic tame filling function for a presentation grows at least as fast as the respective diameter function, leaving open the question of whether or not tame filling functions ever grow strictly faster than their respective diameter functions, especially in the case of groups with linear diameter functions.  In fact, it was not even known if every group admits a finite-valued intrinsic or extrinsic tame filling function, whereas every group admits a finite-valued intrinsic and extrinsic diameter function by definition.  The main purpose of this paper is to prove the following theorem, and thereby resolve these questions for intrinsic tame filling functions:

\begin{customthm}{\ref{main}} Given a finite presentation $\mathcal{P} = \langle A | R \rangle$ such that for all $a \in A$, $a$ is not equal to the identity, there is an intrinsic tame filling function for $\mathcal{P}$ that is equivalent to the intrinsic diameter function of $\mathcal{P}$.
\end{customthm}

As a result, all groups admit a finite-valued intrinsic tame filling function.  This was proven concurrently by Nu'Man and Riley; they showed that, given a finite presentation, if $D:\N \to \N$ is the corresponding Dehn function, then $f(n) = D(e^n)$ is an intrinsic tame filling function for the presentation \cite{TimAnisahResult}.  However, since the Dehn function of a presentation grows at least as fast as, and often strictly faster than, the intrinsic diameter function (up to the equivalence relation defined in Section 2), Theorem~\ref{main} provides a tighter bound on the growth of minimal intrinsic tame filling functions.

Beyond that, Theorem~\ref{main} strengthens the notion of intrinsic diameter.  If $f$ is the intrinsic diameter function of a presentation, then by definition, for any word $w$ of length $n$, there exists a van Kampen diagram for $w$ with intrinsic diameter at most $f(n)$.  However, Theorem~\ref{main} implies the existence of an alternative van Kampen diagram for $w$ that is $g$-tame, for some function $g$ equivalent to $f$.  This provides an additional option with more structure than the diameter restriction, and that may prove more useful for some purposes.

Theorem~\ref{main} also shows that intrinsic tame filling functions do not distinguish between groups with equivalent intrinsic diameter functions, as they were defined, in part, to do.  This motivates us to define a couple of new filling functions that may fulfill the desired purpose.  Two different perspectives on what intrinsic tame filling functions measure provide direction in defining these new functions.  We show that both new functions grow at least as quickly as the intrinsic diameter function.  We leave it to future work to show whether or not they grow strictly faster than intrinsic diameter functions for some groups.

The questions resolved for intrinsic tame filling functions in this paper are still open in the case of extrinsic tame filling functions; it is quite possible that there exists a group whose extrinsic tame filling functions all grow strictly faster than its extrinsic diameter function.  This possibility is supported by a conjecture of Tschantz---that there exists a finitely presented group that is not tame combable \cite[p. 3]{TFFs}.  Brittenham and Hermiller showed that such a group would not even admit a finite-valued extrinsic tame filling function, and therefore it would not have an extrinsic tame filling function that is equivalent to its extrinsic diameter function.  Unfortunately, an argument similar to the one used in this paper would not be sufficient to answer these questions.  See the remark at the end of Section 6 for a brief explanation.

In Section 2, we give preliminary definitions and notation.

In Section 3, we define the \emph{icicles} of a spanning tree of a van Kampen diagram: each edge outside of the spanning tree connects two branches of the tree, forming a loop that bounds the icicle of that edge.  These icicles provide a natural way to choose 1-combings, and the simple ways in which they intersect each other provide a method for finding van Kampen diagrams with increasingly tame 1-combings.  

In Section 4, we construct 1-combings that follow the icicles of a spanning tree outward from the basepoint to the boundary of the diagram and show that they have certain nice properties.  Such 1-combings are particularly tame if the branches of the tree are geodesics, making them useful in the proof of Theorem ~\ref{main}.  

Section 5 is the proof of a technical lemma that provides the main strategy in the proof of Theorem ~\ref{main}.  It allows us to replace the body of an icicle with a diagram that has relatively small diameter while preserving distance in the rest of the diagram.  

Section 6 contains the remainder of the proof of Theorem ~\ref{main}, using the results of sections 3, 4, and 5.

In Section 7, we introduce new filling functions that are potential proper refinements of the intrinsic diameter function.

\section{Preliminary Definitions and Notation}

Most of the definitions from this section are taken from \cite[Chapter II]{FFs} and \cite{TFFs}.

Throughout this paper, let $\mathcal{P} = \langle A | R \rangle$ be a finite presentation of a group $G$, where for all $a \in A$, $a \neq_G 1$.  Let $X$ be the Cayley 2-complex of $\langle A | R \rangle$.  Given a word $w \in (A \cup A^{-1})^*$, let $\ell(w)$ denote the length of $w$.  Given an edge path $\gamma$ in a 2-complex, let $|\gamma|$ denote the length of $\gamma$, that is, the number of edges in $\gamma$.

Given a 2-complex $C$, define a \emph{coarse distance function} $d_C:C^{(0)} \times C \to \N[\frac{1}{4}]$ in the following way.  Let $v \in C^{(0)}$ and $x \in C$.  If $x \in C^{(0)}$, let $d_C(v,x)$ be the length of the shortest edge path between $v$ and $x$.  If $x \in C^{(1)} \setminus C^{(0)}$, then let $p$ and $q$ be the endpoints of the edge on which $x$ lies and let $$d_C(v,x) = \text{min}(d_C(v,p), d_C(v,q)) + \frac{1}{2}.$$  If $x \in C \setminus C^{(1)}$, then let $e_1, \dots, e_n$ be the 1-cells on the boundary of the 2-cell in which $x$ lies and for each $i \in [n]$ choose some $y_i \in e_i$.  Then let $$d_C(v,x) = \text{max}\{ d_C(v,y_i) | i \in [n] \} - \frac{1}{4}.$$

If $C$ is a 2-complex and $x,y \in C^{(0)}$, then a \emph{$C$-geodesic} from $x$ to $y$ is an edge path from $x$ to $y$ in $C^{(1)}$ of length $d_C(x,y)$.  If the complex being referred to is clear, we will just call such a path a geodesic.  If $T$ is a tree in $C$, and $* \in C^{(0)} \cap T$, we will say that $T$ is a \emph{tree of $C$-geodesics out of $*$} if, for every $x \in T^{(0)}$, the unique simple edge path from $*$ to $x$ in $T$ is a $C$-geodesic.  If $D \subseteq C$, we will write $E(D)$ for the set of 1-cells of $C$ contained in $D$.

A \emph{van Kampen diagram} $\Delta$ is a finite, contractible 2-complex embedded in $\R^2$ that has a specified basepoint $* \in \Delta^{(0)} \cap \partial \Delta$, with the following property.  For each 1-cell of $\Delta$, there are two corresponding directed edges, one labeled by an element $a \in A$ and the other labeled by $a^{-1}$, and there is a cellular map $\pi_\Delta:\Delta \to X$ that preserves the directions and labels of these edges.  It is possible to construct an edge circuit $\gamma$ starting and ending at $*$ by traveling around the boundary of $\Delta$ counterclockwise until an edge corresponding to every 1-cell in the boundary has been used.  Note that $\gamma$ may not be unique, since if $*$ is a cut vertex there will be multiple edges that serve to begin the circuit.  Then $\gamma$ is labeled by a word $w \in (A \cup A^{-1})^*$ and $\Delta$ is called a van Kampen diagram for $w$.  Note that $w =_G 1$, since $w$ also labels $\pi_\Delta(\gamma)$ in $X$, which is also an edge circuit.

Given a van Kampen diagram $\Delta$ with basepoint $*$, the \emph{intrinsic diameter} of $\Delta$ is $$\IDiam(\Delta) = \text{max}\{d_{\Delta}(*,x) | x \in \Delta^{(0)}\}.$$

Given $w \in (A \cup A^{-1})^*$ with $w =_G 1$, the \emph{intrinsic diameter} of $w$ is $$\IDiam(w) = \text{min} \{\IDiam(\Delta) | \Delta \text{ is a van Kampen diagram for } w \text{ with basepoint } *\}.$$

\begin{Def}
The \emph{intrinsic diameter function} of $\langle A | R \rangle$ is the function $\IDiam: \N \to \N$ given by $$\IDiam(n) = \text{max} \{\IDiam(w) | w \in (A \cup A^{-1})^* \text{ with } \ell(w) \leq n \text{ and } w =_G 1 \}.$$
\end{Def}

Note that in \cite[Chapter II]{FFs}, intrinsic diameter is defined using distance between arbitrary vertices, rather than fixing one of them as the basepoint, and the function defined above is referred to as the \emph{based} intrinsic diameter.  However, these two definitions give equivalent invariants, and the above definition is more convenient for the purposes of this work.

\begin{Def} \label{equivalent definition} Let $S,T \subseteq \N[\frac{1}{4}]$ and let $f:S \to [0,\infty)$ and $g:T \to [0,\infty)$ be functions.  Write $f \preceq g$ if there exist constants $A, B, C, D, E \ge 0$ such that for all $s \in S$, $f(s) \leq Ag(Bt+C) + Dt + E$, where $t = \text{max}([0,s] \cap T)$.  If $f \preceq g$ and $g \preceq f$, write $f \simeq g$ and say that $f$ and $g$ are \emph{equivalent}.  If $P$ is a property of functions and $g$ has $P$, say that $f$ \emph{is at most $P$} if $f \preceq g$.  For example, $f$ is at most linear if $f \preceq g$ and $g$ is linear.
\end{Def}

Note that this notion of equivalence can distinguish between, for example, linear, quadratic, polynomial, and exponential growth, among others.  Using this equivalence relation, the equivalence class of a filling function of a finite presentation is a quasi-isometry invariant; see \cite[Chapter II]{FFs} for details.  Although all the filling functions listed in \cite[Chapter II]{FFs} have a domain of $\N$, the above definition includes functions with domains in $\N[\frac{1}{4}]$ specifically in order to compare tame filling functions with diameter functions.

The following definition was introduced by Mihalik and Tschantz in \cite{1-combings}:

Given a 2-complex $C$ with basepoint $*$ and a subcomplex $D \subseteq C^{(1)}$, a \emph{1-combing} of the pair $(C,D)$ based at $*$ is a continuous function $\Psi: D \times I \to C$ such that 
\begin{itemize}
\item for all $t \in I$, $\Psi(*,t) = *$,
\item for all $x \in D$, $\Psi(x,0) = 0$ and $\Psi(x,1) = x$, and
\item for all $p \in D^{(0)}$ and all $t \in I$, $\Psi(p,t) \in C^{(1)}$.
\end{itemize}

Let $f: \N[\frac{1}{4}] \to \N[\frac{1}{4}]$ be a function.  A 1-combing $\Psi$ of the pair $(C,D)$ is \emph{$f$-tame} if for all $n \in \N[\frac{1}{4}]$ and for all $x \in D$ and $s,t \in [0,1]$ such that $s \leq t$, if $d_{C}(*,\Psi(x,t)) \leq n$, then $d_{C}(*,\Psi(x,s)) \leq f(n)$.

\begin{Def} Let $f: \N[\frac{1}{4}] \to \N[\frac{1}{4}]$ be non-decreasing.  $f$ is an \emph{intrinsic tame filling function} for $\langle A | R \rangle$ if, for all $w \in (A \cup A^{-1})^*$ with $w =_G 1$, there is a van Kampen diagram $\Delta_w$ for $w$ with basepoint $*$ and a 1-combing $\Psi_w$ of $(\Delta_w, \partial \Delta_w)$ based at $*$ such that $\Psi_w$ is $f$-tame.
\end{Def}

\section{Icicles}
Let $\Delta$ be a van Kampen diagram and $T$ a spanning tree of $\Delta$.  Given a 1-cell $e \in E(\Delta \setminus T)$ with a corresponding directed edge $\overrightarrow{e}$ directed from a vertex $x$ to a vertex $y$, let $\gamma_x$ and $\gamma_y$ be the unique simple paths in $T$ from $*$ to $x$ and $y$, respectively.  Let $\alpha$ be the longest initial segment on which $\gamma_x$ and $\gamma_y$ agree, and define $\beta_x$ and $\beta_y$ by $\gamma_x = \alpha \cdot \beta_x$ and $\gamma_y = \alpha \cdot \beta_y$.  Then $\eta = \beta_x \cdot \overrightarrow{e} \cdot \overline{\beta_y}$ is a simple circuit.  Since $\Delta$ is planar, by the Jordan Curve Theorem, we know that $\eta$ splits the plane into two components, the inside and outside of $\eta$.  See Figure~\ref{fig:Icicle Definitions}.

\begin{figure}[h!]
    \includegraphics[width=.8\linewidth]{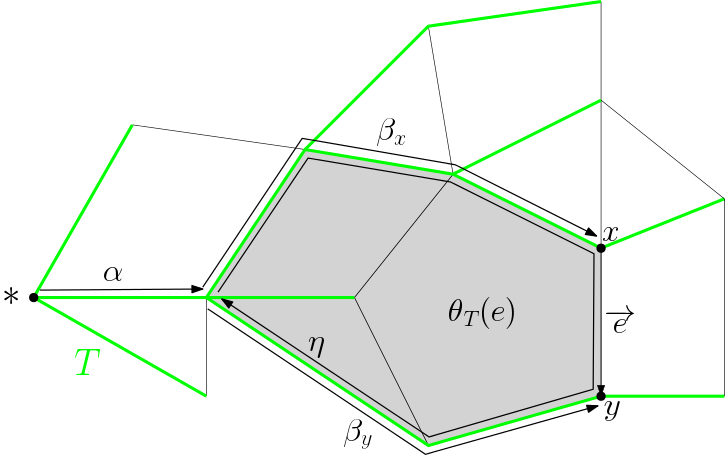}
    \caption{The $T$-icicle at $e$.  Edges in the tree $T$ are thickened.  The body of the icicle is shaded and $\alpha$ is the tail.}
    \label{fig:Icicle Definitions}
\end{figure}

\begin{Def} Given the above notation, the \emph{$T$-icicle at $e$} is the union of $\alpha$, $\eta$, and the inside of $\eta$.  It will just be called the icicle at $e$ and be denoted $I_e$ if the tree being used is clear.  Define the \emph{tail} of the icicle to be $\alpha$ and the \emph{body} of the icicle to be the union of $\eta$ and the inside of $\eta$.
\end{Def}

\begin{Def} Let $\theta_T(e)$ be the 2-cell in the body of $I_e$ with $e \subset \partial \theta_T(e)$.  Then $\theta_T$ is a function from $E(\Delta \setminus T)$ to the 2-cells of $\Delta$, which we will call the \emph{icicle flow function} for $T$.
\end{Def}

Icicles are the building blocks we will use to construct van Kampen diagrams with tame 1-combings, and are also the basis for constructing said 1-combings.  Icicles are particularly useful tools for these purposes in part because of the way they intersect each other: for every pair of icicles, either one is entirely contained in the other or they do not intersect, except possibly at their boundary.

\begin{Lemma} \label{Icicle relationships} Let $\Delta$ be a van Kampen diagram and $T$ a spanning tree of $\Delta$.  Let $e,e' \in E(\Delta \setminus T)$.  If $e' \subseteq I_e$, then $I_{e'} \subseteq I_e$.  If $e' \not \subseteq I_e$ and $e \not \subseteq I_{e'}$, then $\interior{I_e} \cap \interior{I_{e'}} = \emptyset$.
\end{Lemma}

\begin{proof}

Let $e$ have endpoints $x$ and $y$ and $e'$ have endpoints $x'$ and $y'$.  Let $\gamma_x$, $\gamma_y$, $\gamma_{x'}$, and $\gamma_{y'}$ be the unique simple paths in $T$ from $*$ to $x$, $y$, $x'$, and $y'$, respectively.  Suppose $e' \subseteq I_e$.  We will first show that $\gamma_{x'}$ and $\gamma{y'}$ are contained in $I_e$.  Suppose by way of contradiction that there is a point $p \in \gamma_{x'} \setminus I_e$.  Since $x' \in I_e$, there is also some vertex $q \in \gamma_{x'} \cap \partial I_e \cap \Delta^{(0)}$ such that $p$ is between $*$ and $q$ in $\gamma_{x'}$.  Since $\partial I_e \cap \Delta^{(0)} \subset \gamma_x \cup \gamma_y$, without loss of generality, let $q \in \gamma_x$.  Then the unique simple path in $T$ from $*$ to $q$ must be contained in $\gamma_{x}$. But since $p$ is between $*$ and $q$ in $\gamma_{x'}$, this path must also contain $p$.  This contradicts the fact that $p \not \in I_e$.  So $\gamma_{x'} \subset I_e$.  By the same argument, $\gamma_{y'} \subset I_e$.  So $\partial I_{e'} = \gamma_{x'} \cup \gamma_{y'} \cup e' \subset I_e$.  Since $I_e$ is simply connected, this implies that $I_{e'} \subseteq I_e$.

Now suppose $e' \not \subseteq I_e$ and $e \not \subseteq I_{e'}$.  Let $B_e$ and $B_{e'}$ be the bodies of $I_e$ and $I_{e'}$, respectively.  Note that $\interior{I_e} = \interior{B_e}$ and $\interior{I_{e'}} = \interior{B_{e'}}$.  I will first show that $\gamma_{x'}$ does not intersect $\interior{B_e}$.  For suppose there were some $p \in \gamma_{x'} \cap \interior{B_e}$.  Now since $e' \not \subseteq I_e$, $e' \cap \interior{B_e} = \emptyset$.  In particular, $x' \not \in \interior{B_e}$.  So there is some vertex $q \in \partial B_e \cap \Delta^{(0)}$ such that $p$ is between $*$ and $q$ in $\gamma_{x'}$.  Since $\partial B_e \cap \Delta^{(0)} \subset \gamma_x \cup \gamma_y$, without loss of generality let $q \in \gamma_x$.  Then the unique simple path in $T$ from $*$ to $q$ must be contained in $\gamma_{x}$. But since $p$ is between $*$ and $q$ in $\gamma_{x'}$, this path must also contain $p$.  This contradicts the fact that $p \in \interior{B_e}$, since $\interior{B_e}$ does not intersect $\gamma_x$.  So $\gamma_{x'} \cap \interior{B_e} = \emptyset$.  By the same argument, $\gamma_{y'} \cap \interior{B_e} = \gamma_{x} \cap \interior{B_{e'}} = \gamma_{y} \cap \interior{B_{e'}} = \emptyset$.  Then since $\partial B_e \subseteq e \cup \gamma_x \cup \gamma_y$, we have that $\partial B_e \cap \interior{B_{e'}} = \emptyset$.  In the same way, $\partial B_{e'} \cap \interior{B_e} = \emptyset$.

Now consider $C = \interior{B_e} \cap \interior{B_{e'}}$, and suppose by way of contradiction that $C \neq \emptyset$.  Note that in this case, $\bar{C} = B_e \cap B_{e'}$ and that $\partial C \subseteq \partial B_e \cup \partial B_{e'}$.  Suppose there is some $z \in \partial C \setminus \partial B_e$.  Then $z \in \partial B_{e'}$ and since $z \in \bar{C} \setminus \partial B_e$, we must have that $z \in \interior{B_e}$.  This contradicts that $\partial B_{e'} \cap \interior{B_e} = \emptyset$, so it must be that case that $\partial C \setminus \partial B_e = \emptyset$.  By the same argument, $\partial C \setminus \partial B_{e'} = \emptyset$.  Therefore, $\partial C = \emptyset$.  This implies that $\bar{C} = \interior{C}$, which makes $C = \emptyset$.

\end{proof}

\section{1-combings Respecting a Spanning Tree} \label{1-combings section}

In this section, we will construct a special type of 1-combing that will end up being particularly tame.  This is because the combing paths will respect a spanning tree, growing outwards from the basepoint of the diagram along the icicles of the tree and never crossing a branch of the tree.  If the spanning tree chosen is a tree of geodesics out of the basepoint, then such a 1-combing has combing paths that are in some way close to being geodesics themselves, resulting in a 1-combing that is about as tame as possible.  

Given a 2-complex $C$, a subcomplex $D \subseteq C^{(1)}$, a point $x \in C$, and a 1-combing $\Psi$ of the pair $(C,D)$, define 
$$P(\Psi,x) = \{ y \in C | \text{there is } p \in D \text{ and } 0 \leq s \leq t \leq 1 \text{ with } \Psi(p,t) = x \text{ and } \Psi(p,s) = y \},$$ 
the set of points prior to $x$ in $\Psi$.  This is the set of points $y$ such that there is some combing path containing $x$ where $y$ appears on that path before $x$.

\begin{Def} Given a van Kampen diagram $\Delta$ with basepoint $*$ and a spanning tree $T$ of $\Delta$, a \emph{1-combing respecting $T$} is a 1-combing $\Psi$ of the pair $(\Delta, \partial \Delta)$ based at $*$ satisfying two properties:
\begin{enumerate}
\item For all $e \in E(\partial \Delta \setminus T)$, $\bigcup_{x \in e} P(\Psi,x) = I_e$.
\item For $p \in \partial \Delta$ and $t \in [0,1]$, if $\Psi(p,t) \in T$, then $\Psi(p,[0,t])$ is the unique simple path in $T$ from $*$ to $\Psi(p,t)$.
\end{enumerate}
\end{Def}

The main purpose of this section is to show that such 1-combings exist:

\begin{Prop} \label{Nice 1-combings exist}
Let $\Delta$ be a van Kampen diagram and let $T$ be a spanning tree of $\Delta$.  Then there exists a 1-combing $\Psi_T$ that respects $T$.
\end{Prop}

\begin{proof}

We will first demonstrate how to construct such a 1-combing, and then prove that it is a 1-combing that respects the given tree.  Figure~\ref{fig:Psi_T} shows an example of what the combing paths of our constructed 1-combings could look like.

\begin{figure}[h!]
    \includegraphics[width=.9\linewidth]{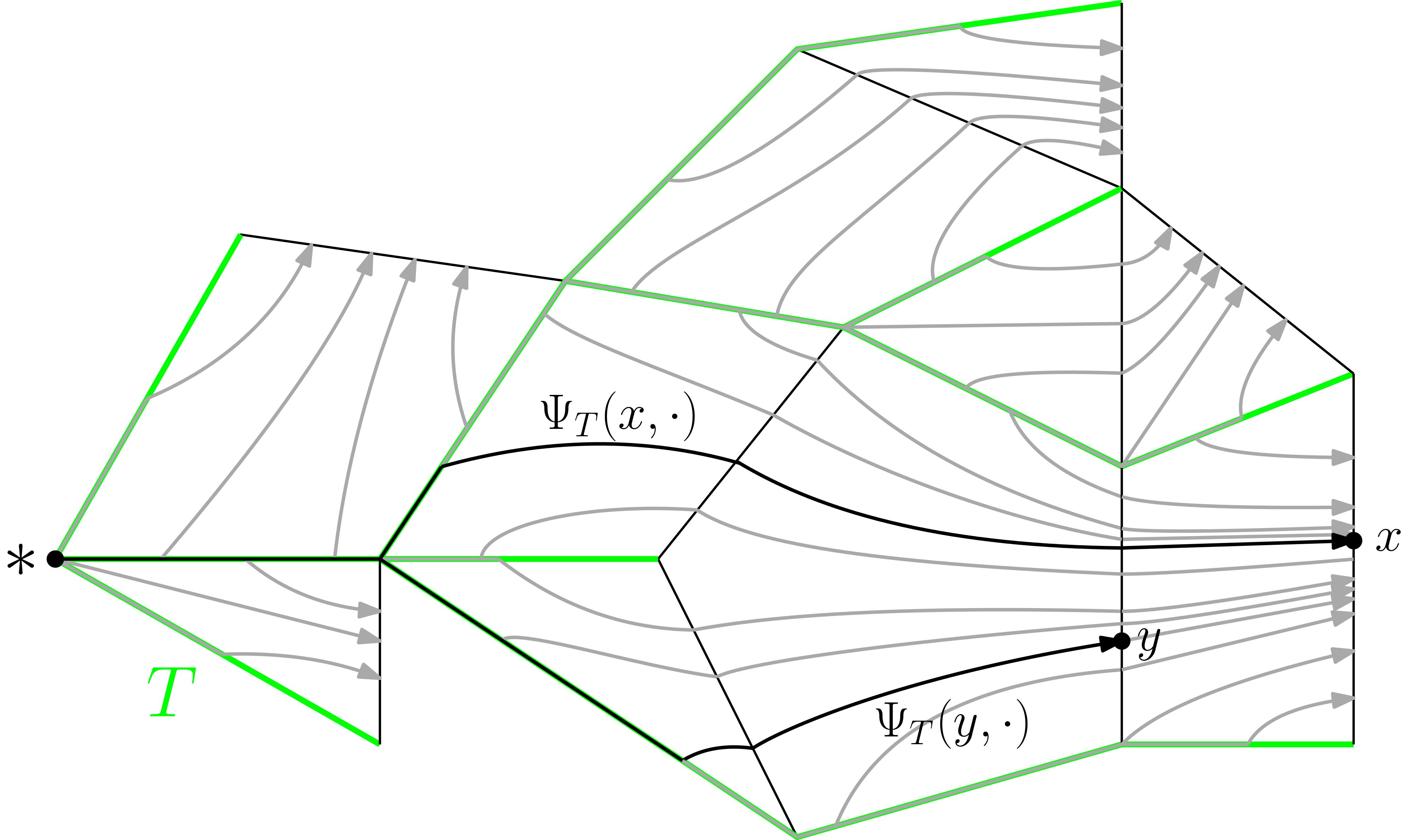}
    \caption{Representation of $\Psi_T$}
    \label{fig:Psi_T}
\end{figure}

We will define our combing paths in segments, one for each 2-cell, starting from the boundary of the diagram and working inwards towards the basepoint, following the icicles.  At each step, we identify a 1-cell $e \in E(\Delta \setminus T)$ that is on the boundary of the subcomplex on which combing paths have not yet been defined, as well as its associated 2-cell $\sigma = \theta_T(e)$, the one that is inside the icicle at $e$ with $e \subseteq \partial \sigma$.  Then we determine the segments of the combing paths that intersect $\interior{\sigma}$ with a homotopy $H$ from $e$ to $\partial \sigma \setminus \interior{e}$.  Note that the combing paths will run in the opposite direction to the paths $H(x,\cdot)$ for each $x \in e$.  Once all the segments have been defined, we stitch them together into a continuous 1-combing $\Psi_T$.

Let $\Delta$ be a van Kampen diagram and $T$ a spanning tree of $\Delta$.  Let $\Delta_0 = \Delta$, and define $\psi_{0} : \partial \Delta \times [0,1] \to \Delta$ by $\psi_{0}(x,t) = x$.  Now suppose we have defined a function $\psi_{i}: \partial \Delta \times [0,1] \to \Delta$ and a subcomplex $\Delta_i$ of $\Delta$ with the following four properties:
\begin{itemize}
\item $T \subseteq \Delta_i$.
\item $\R^2 \setminus \Delta_i$ is connected.
\item $\psi_i$ is continuous.
\item For all $x \in \partial \Delta$, $\psi_i(x,0) \in \Delta_i^{(1)}$.
\end{itemize}

Note that these properties all hold for $i = 0$.  Now if all 1-cells of $\partial \Delta_i$ are in $T$, the process ends.  Otherwise, we will define $\psi_{i+1}$ and $\Delta_{i+1}$ such that these same properties hold, as in Figure~\ref{fig:psi_{i+1} def}.  Let $e_{i+1}$ be a 1-cell of $\partial \Delta_i$ that is not in $T$, and let $I_{i+1}$ be the $T$-icicle at $e_{i+1}$.  Let $\sigma_{i+1} = \theta_T(e_{i+1})$, the 2-cell in $I_{i+1}$ whose boundary contains $e_{i+1}$.  Note that $\partial I_{i+1} \subseteq \Delta_i$, so since $\R^2 \setminus \Delta_i$ is connected, we must have that $I_{i+1} \subseteq \Delta_i$.  Therefore, $\sigma_{i+1} \subseteq \Delta_i$.  Let $\Delta_{i+1} = \Delta_i \setminus (\interior{\sigma_{i+1}} \cup \interior{e_{i+1}})$.  Since $\Delta_{i+1}$ is a subcomplex of $\Delta_i$ and we have not removed any vertices or any 1-cells in $T$ from $\Delta_i$, $\Delta_{i+1}$ also contains $T$.  Additionally, since $e_{i+1}$ is on the boundary of $\Delta_i$, $$\R^2 \setminus \Delta_{i+1} = (\R^2 \setminus \Delta_i) \cup \interior{\sigma_{i+1}} \cup \interior{e_{i+1}}$$ is connected.

We will define a continuous $H_{i+1} : e_{i+1} \times [0,1] \to \sigma_{i+1}$ with the following properties: 
\begin{enumerate}[i)]
\item For all $x \in e_{i+1}$, $H_{i+1}(x,0) = x$ and $H_{i+1}(x,1) \in \partial \sigma_{i+1} \setminus \interior{e_{i+1}}$.
\item If $x$ is an endpoint of $e_{i+1}$, then $H_{i+1}(x,t) = x$ for all $t \in [0,1]$.
\item For all $x \in \interior{e_{i+1}}$ and all $t \in (0,1)$, $H_{i+1}(x,t) \in \interior{\sigma_{i+1}}$.
\end{enumerate}

Since $\sigma_{i+1}$ is a 2-cell, we know that it is homeomorphic to the unit disk $$D^2 = \{(a,b) \in \R^2 | a^2 + b^2 \leq 1 \}.$$  Define $$S^1_+ = \{ (a,b) \in \R^2 | a^2 + b^2 = 1 \text{ and } b \ge 0 \} \text{ and } S^1_- = \{ (a,b) \in \R^2 | a^2 + b^2 = 1 \text{ and } b \leq 0 \}.$$  Then let $\phi_{\sigma_{i+1}}: \sigma_{i+1} \to D^2$ be a homeomorphism such that $\phi_{\sigma_{i+1}}(e_{i+1}) = S^1_+$.  Note that this implies that $\phi_{\sigma_{i+1}}(\partial \sigma_{i+1} \setminus \interior{e_{i+1}}) = S^1_-$.  Then define $h: S^1_+ \times [0,1] \to D^2$ by $$h((a,b),t) = (a,(1-2t)b).$$  Note that $h$ is a straight-line homotopy rel boundary from $S^1_+$ to $S^1_-$.  Finally, let $$H_{i+1} = \phi_{\sigma_{i+1}}^{-1} \circ h \circ (\phi_{\sigma_{i+1}} |_{e_{i+1}} \times \text{id}_{[0,1]}).$$
This definition gives the desired properties, for:
\begin{enumerate}[i)]
\item Let $x \in e_{i+1}$.  Then $$H_{i+1}(x,0) = \phi_{\sigma_{i+1}}^{-1}(h(\phi_{\sigma_{i+1}}(x),0)) = \phi_{\sigma_{i+1}}^{-1}(\phi_{\sigma_{i+1}}(x)) = x.$$  Also, since $\phi_{\sigma_{i+1}}(x) \in S^1_+$, this implies that $h(\phi_{\sigma_{i+1}}(x),1) \in S^1_-$, so $$H_{i+1}(x,1) = \phi_{\sigma_{i+1}}^{-1}(h(\phi_{\sigma_{i+1}}(x),1)) \in \phi_{\sigma_{i+1}}^{-1}(S^1_-) = \partial \sigma_{i+1} \setminus \interior{e_{i+1}}.$$
\item Let $x$ be an endpoint of $e_{i+1}$ and let $t \in [0,1]$.  Then $\phi_{\sigma_{i+1}}(x) \in \{(1,0),(-1,0)\}$.  Since the vertical component of $\phi_{\sigma_{i+1}}(x)$ is 0, $h(\phi_{\sigma_{i+1}}(x),t) = \phi_{\sigma_{i+1}}(x)$.  Thus, $$H_{i+1}(x,t) = \phi_{\sigma_{i+1}}^{-1}(h(\phi_{\sigma_{i+1}}(x),t)) = \phi_{\sigma_{i+1}}^{-1}(\phi_{\sigma_{i+1}}(x)) = x.$$
\item Let $x \in \interior{e_{i+1}}$ and $t \in (0,1)$.  Let $\phi_{\sigma_{i+1}}(x) = (a,b)$.  Since $x \in \interior{e_{i+1}}$, we know that $b > 0$.  Since $t \in (0,1)$, this implies that $|(1-2t)b| < |b|$.  Therefore, $a^2 + ((1-2t)b)^2 < 1$, which means that $h(\phi_{\sigma_{i+1}}(x),t) \in \interior D^2$.  Hence, $$H_{i+1}(x,t) = \phi_{\sigma_{i+1}}^{-1}(h(\phi_{\sigma_{i+1}}(x),t)) \in \interior{\sigma_{i+1}}.$$
\end{enumerate}

Now define $\psi_{i+1}: \partial \Delta \times [0,1] \to \Delta$ by 
$$\psi_{i+1}(x,t) = 
\begin{cases}
H_{i+1}(\psi_i(x,0),1-t), & \psi_i(x,0) \in e_{i+1}\\
\psi_i(x,0), & \psi_i(x,0) \in \Delta^{(1)} \setminus \interior{e_{i+1}}.
\end{cases}$$

\begin{figure}[t]
    \includegraphics[width=.9\linewidth]{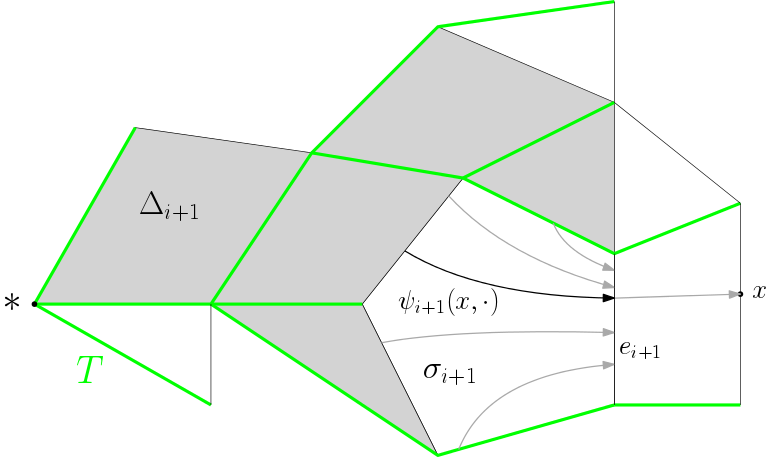}
    \caption{Defining $\Delta_{i+1}$ and $\psi_{i+1}$.  The entire diagram is $\Delta$.  $\Delta_{i+1}$ is the shaded region along with $T$, and $\Delta_i = \Delta_{i+1} \cup \sigma_{i+1}$.}
    \label{fig:psi_{i+1} def}
\end{figure}

Note that the two pieces of $\psi_{i+1}$ agree on their intersection because $H_{i+1}$ fixes the endpoints of $e_{i+1}$, so $\psi_{i+1}$ is well-defined and, by the pasting lemma, continuous.

Finally, let $x \in \partial \Delta$.  We know that $\psi_i(x,0) \in \Delta_i^{(1)}$, so if $\psi_i(x,0) \not \in e_{i+1}$, then $\psi_{i+1}(x,0) = \psi_i(x,0) \in \Delta_{i+1}^{(1)}$.  But if $\psi_i(x,0) \in e_{i+1}$, then $$\psi_{i+1}(x,0) = H_{i+1}\left(\psi_i(x,0),1\right) \in \partial \sigma_{i+1} \setminus \interior{e_{i+1}} \subseteq \Delta_{i+1}^{(1)}.$$  In either case, $\psi_{i+1}(x,0) \in \Delta_{i+1}^{(1)}$.  So $\Delta_{i+1}$ and $\psi_{i+1}$ satisfy the desired properties.

Now, since there are only finitely many 1-cells of $\Delta$ outside $T$, this construction must end with some $\Delta_n$ such that $\partial \Delta_n \subseteq T$.  I claim that, in fact, $\Delta_n = T$.  We know that $T \subseteq \Delta_n$.  Suppose by way of contradiction that there exists $x \in \Delta_n \setminus T$.  Then in particular $x \in \interior{\Delta_n}$, which means that $\R^2 \setminus \partial \Delta_n$ is disconnected, and therefore $\R^2 \setminus T$ is disconnected.  This is a contradiction, since a finite, planar tree cannot disconnect the plane, so $\Delta_n = T$.

Now let $H_{n+1}: T \times [0,1] \to T$ be the homotopy from $T$ to $*$ such that, for each $x \in T$, $H_{n+1}(x,\cdot)$ is the constant speed parametrization of the unique simple path from $x$ to $*$ in $T$.  Note that for all $x \in \partial \Delta$, $\psi_n(x,0) \in \Delta_n = T$, so we can define $\psi_{n+1}: \partial \Delta \times [0,1] \to \Delta$ by $\psi_{n+1}(x,t) = H_{n+1}(\psi_n(x,0),1-t)$.

Finally, define $\Psi_T: \partial \Delta \times [0,1] \to \Delta$ by $\Psi_T(x,t) = \psi_i(x,(n+1)t+i-n-1)$ for $t \in \left[1 - \frac{i}{n+1}, 1 - \frac{i-1}{n+1}\right]$.  Note that when $t = 1 - \frac{i}{n+1}$, we have $(n+1)t+i-n-1 = 0$ and when $t = 1 - \frac{i-1}{n+1}$, we have $(n+1)t+i-n-1 = 1$.  So to show that $\Psi_T$ is well-defined, we must show that for all $i \in [n]$ and $x \in \partial \Delta$, $\psi_{i+1}(x,1) = \psi_i(x,0)$.  This follows immediately from the definitions of $\psi_{i+1}$ and the fact that $H_{i+1}(y,0) = y$ for all $y \in e_{i+1}$.  Since each $\psi_i$ is continuous, $\Psi_T$ is continuous by the pasting lemma.

We will now establish that $\Psi_T$ is a 1-combing that respects $T$.  It will be convenient to first establish the second property, that for all $p \in \partial \Delta$ and $t \in [0,1]$, if $\Psi(p,t) \in T$, then $\Psi(p,[0,t])$ is the unique simple path in $T$ from $*$ to $\Psi(p,t)$.

Let $p \in \partial \Delta$ and $t \in [0,1]$, and suppose that $x = \Psi(p,t) \in T$.  There is some $i \in [n+1]$ such that $t \in \left[1 - \frac{i}{n+1}, 1 - \frac{i-1}{n+1}\right]$, so $x = \psi_i(p,s)$ where $s = (n+1)(t-1)+i$.  If $t \leq \frac{1}{n+1}$, we may choose $i = n+1$ so that $x = \psi_{n+1}(p,s)$.  Then $\Psi(p,[0,t]) = \psi_{n+1}(p,[0,s])$, which is by definition the unique simple path in $T$ from $*$ to $x$.  Otherwise, we may choose $i \leq n$ such that $t \neq 1 - \frac{i}{n+1}$, and therefore $s \neq 0$.  In that case, if $s = 1$, then $x = \psi_{i-1}(p,0)$ from the definition of $\psi_i$.  If $0 < s < 1$, we cannot have that $x = H_i(\psi_{i-1}(x,0),1-s)$, since $x \in T \subseteq \Delta^{(1)}$ and we know that $H_i(\psi_{i-1}(x,0),1-s) \not \in \Delta^{(1)}$ because $1-s \in (0,1)$.  So from the definition of $\psi_i$ we must have that $x = \psi_{i-1}(p,0) \in T$.  Therefore, $\psi_{i-1}(p,0) \not \in \interior{e_i}$, so $\psi_i(p,r) = \psi_{i-1}(p,0) = x$ for all $r \in [0,1]$.  Now for all $j \in [n]$ with $j > i$, suppose by induction that $\psi_{j-1}(p,0) = x$.  Then by the same logic, $\psi_{j-1}(p,0) \not \in \interior{e_j}$, so $\psi_j(p,r) = x$ for all $r \in [0,1]$.  As a result, $\Psi_T\left(p,\left[\frac{1}{n+1}, t\right]\right) = \{ x \}$.  So $\Psi_T(p,[0, t]) = \Psi_T\left(p,\left[0,\frac{1}{n+1}\right]\right)$, and we already know from the case with $t \leq \frac{1}{n+1}$ that $\Psi_T\left(p,\left[0,\frac{1}{n+1}\right]\right)$ is the unique simple path from $*$ to $x$.  This establishes the desired property.

Now we will show that $\Psi_T$ is in fact a 1-combing.  We already know that $\Psi_T$ is continuous.  For any $p \in \partial \Delta$, $\Psi_T(p,0) = \psi_{n+1}(p,0) = *$ and $\Psi_T(p,1) = \psi_1(p,1) = \psi_0(p,0) = p$.  Also, $p \in T$, so $\Psi(p,1) \in T$.  Then by the second property of a 1-combing respecting $T$, we know that $\Psi(p,[0,1]) \subseteq T \subseteq \Delta^{(1)}$.  In particular, if $p = *$, the unique simple path from $*$ to $*$ in $T$ is the constant path, so $\Psi(*,[0,1]) = \{ * \}$.  Hence, $\Psi_T$ is a 1-combing.

Finally, we will show that $\Psi_T$ satisfies the first property of a 1-combing respecting $T$, that for all $e \in E(\partial \Delta \setminus T)$, $\bigcup_{x \in e} P(\Psi,x) = I_e$.

First we will show that $I_e \supseteq \bigcup_{x \in e} P(\Psi_T,x)$.  Let $p \in \partial \Delta$ and $t \in [0,1]$ such that $\Psi_T(p,t) \in e$.  Suppose by way of contradiction that there is some $s \leq t$ so that $\Psi_T(p,s) \not \in I_e$.  Then there is some smallest $s' \in [s,t]$ such that $\Psi_T(p,s') \in I_e$.  Note that this implies that $\Psi_T(p,s')$ is on the boundary of the icicle.  There are two possible cases for where $\Psi_T(p,s')$ lands, as shown in Figure~\ref{fig:P in I}.

\begin{figure}[h!]
  \centering
  \begin{subfigure}[b]{.4\linewidth}
    \includegraphics[width=\linewidth]{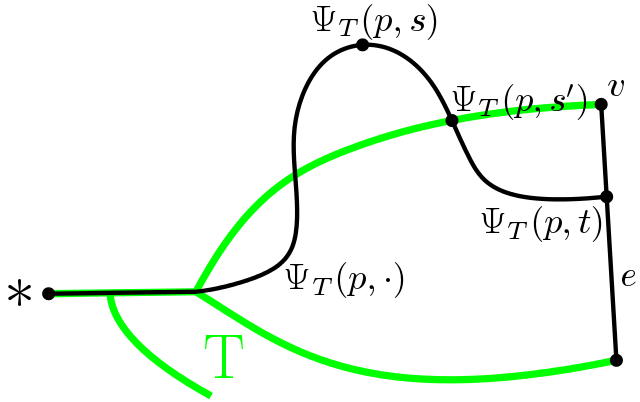}
    \caption{Case 1}
  \end{subfigure}
  \hspace{.05\linewidth}
  \begin{subfigure}[b]{.5\linewidth}
    \includegraphics[width=\linewidth]{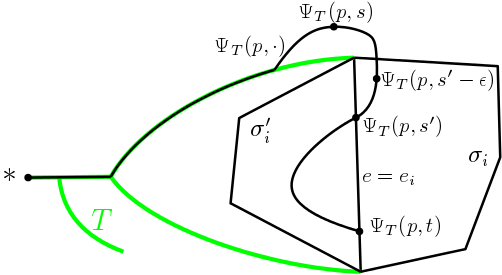}
    \caption{Case 2}
  \end{subfigure}
  \caption{If a prior point of $e$ in $\Psi_T$ were outside the icicle at $e$}
  \label{fig:P in I}
\end{figure}

Case 1: $\Psi_T(p,s') \in T$.  By the definition of $I_e$, $\Psi_T(p,s')$ is on the simple path in $T$ from $*$ to some endpoint $v$ of $e$.  Therefore, the simple path in $T$ from $*$ to $\Psi_T(p,s')$ is in the icicle at $e$.  So $\Psi_T(p,[0,s']) \subseteq I_e$.  So $\Psi_T(p,s) \in I_e$, which is a contradiction.  

Case 2: $\Psi_T(p,s') \not \in T$.  Since $\Psi_T(p,s')$ is on the boundary of the icicle at $e$, this implies that $\Psi_T(p,s') \in \interior e$.  Let $\epsilon = \frac{1}{2(n+1)}$ and let $e = e_i$ from the construction of $\Psi_T$.  Then since $\Psi_T(p,s') \in \Delta^{(1)}$, we must have that $\Psi_T(p,s') = \psi_j(p,0)$ for some $j$.  In fact, $j < i$, since we know that $\psi_j(p,0) \in \Delta_j$ but $\interior{e_i} \not \in \Delta_k$ for $k \ge i$.  Since for all $k$ with $j \leq k < i$, if $\psi_k(p,0) \in e_i \neq e_k$ then $\psi_{k+1}(p,0) = \psi_k(p,0)$, we know that $\psi_i(p,1) = \psi_{i-1}(p,0) = \psi_j(p,0) = \Psi_T(p,s')$.  Then $\Psi_T(p,s'-\epsilon) = \psi_i(p,\frac{1}{2}) \in \sigma_i$ by the definition of $H_i$.  Now since $s'$ was the smallest element of $[s,t]$ in the icicle at $e_i$, either $s \leq s'-\epsilon < s'$, in which case $\Psi_T(p,s'-\epsilon)$ is not in the icicle at $e_i$, or $s'-\epsilon < s < s'$, in which case $\Psi_T(p,s) = \psi_i(p,r)$ for some $r \in (\frac{1}{2},1)$, and thus $\Psi_T(p,s) \in \sigma_i$.  Either case implies that $\sigma_i$ is not in the icicle at $e_i$, which contradicts the definition of $\sigma_i$.  This final contradiction completes case 2, implying that no such $s$ exists.

\begin{figure}[h!]
    \includegraphics[width=.7\linewidth]{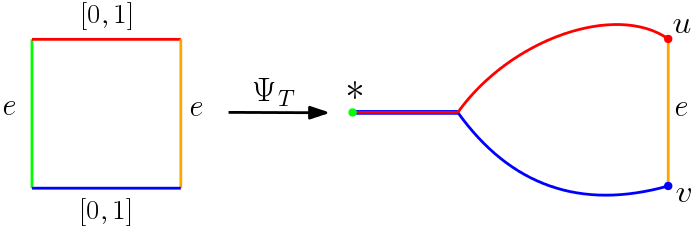}
    \caption{$\Psi_T$ mapping the boundary of $e \times [0,1]$ to the boundary of $I_e$ when $e \subseteq \partial \Delta$.}
    \label{fig:I in P}
\end{figure}

Now we will show the opposite containment.  Let $y \in I_e$.  

Case 1: $e \subseteq \partial \Delta$.  See Figure~\ref{fig:I in P}.  Let $u$ and $v$ be the endpoints of $e$.  Now $\Psi_T|_{e \times [0,1]}$ is a continuous function from the ball $e \times [0,1]$ into $\Delta$.  $\Psi_T(e \times \{0\}) = *$, $\Psi_T(\{u\} \times [0,1])$ is the simple path from $*$ to $u$ in $T$, $\Psi_T(e \times \{1\}) = e$, and $\Psi_T(\{v\} \times [0,1])$ is the simple path from $*$ to $v$ in $T$.  So the image of the path around the boundary of $e \times [0,1]$ is the path around the boundary of the icicle at $e$.  Therefore, since $\Psi_T|_{e \times [0,1]}$ is continuous, every point in the icicle at $e$ is in $\Psi_T(e \times [0,1])$.  Hence, there is some $x \in e$ and $s \in [0,1]$ such that $y = \Psi_T(x,s)$.  Of course, $x = \Psi_T(x,1)$, so $y \in P(\Psi_T,x)$.

Case 2: $e \not \subseteq \partial \Delta$.  First, note that $$\Delta = \bigcup_{f \in E(\partial \Delta \setminus T)} I_f,$$ so for each $x \in \Delta$, there exists a 1-cell $f \subseteq \partial \Delta$ with $x \in I_f$.  Then as a result of Case 1, there must be some $p \in \partial \Delta$ and $s \in [0,1]$ such that $y = \Psi_T(p,s)$.  Since $y \in I_e$, there is some $t \ge s$ such that $\Psi_T(p,t) \in \partial I_e$.  If $\Psi_T(p,t) \in e$, this shows that $y \in P(\Psi_T,x)$, where $x = \Psi_T(p,t) \in e$.  Otherwise, we must have that $\Psi_T(p,t) \in T$.  As a result, $\Psi_T(p,[0,t]) \subseteq T$, and in particular $y \in T$.  Since $y \in I_e$, $y$ must be on the unique simple path in $T$ from $*$ to an endpoint of $e$; call this endpoint $x$.  By the same reasoning as for $y$, there must be some $p' \in \partial \Delta$ and $t' \in [0,1]$ such that $x = \Psi_T(p',t')$.  Then since $x \in T$ and $y$ is on the simple path from $*$ to $x$ in $T$, there must be some $s' < t'$ such that $y = \Psi_T(p',s')$.  Hence, $y \in P(\Psi_T,x)$.
\end{proof}

Note that, in the construction of $\Psi_T$, although it was convenient to deal with the 1-cells of $\Delta \setminus T$ in a somewhat arbitrary order to construct a continuous $\Psi_T$, fewer choices are required to determine the images of individual combing paths.  Since for each $i \in [n]$, $\sigma_i \in \Delta_{i-1} \setminus \Delta_i$, we have that the $\sigma_i$'s are all distinct; in other words, the icicle flow function $\theta_T$ is injective.  Since $\Delta_n = T$ and does not contain any 2-cells, we have that $\theta_T$ is a bijection between the 1-cells of $\Delta \setminus T$ and the 2-cells of $\Delta$.  This bijection, along with a choice of homeomorphism $\phi_\sigma: \sigma \to D^2$ for each 2-cell $\sigma$, determines the images of the combing paths of $\Psi_T$.

\begin{figure}[h!]
    \includegraphics[width=.9\linewidth]{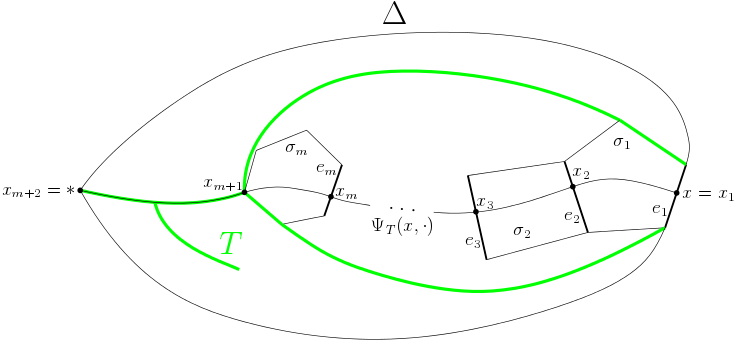}
    \caption{The decomposition of $\Psi_T(x,\cdot)$ given by Lemma~\ref{paths improved}.}
    \label{fig:what paths are}
\end{figure}

In particular, the following lemma will be useful when dealing with individual combing paths.  It decomposes a combing path into an initial segment in the tree $T$ and subsequent segments in each 2-cell that it crosses, as shown in Figure~\ref{fig:what paths are}.

\begin{Lemma} \label{paths improved}
Let $\Delta$ be a van Kampen diagram, $T$ a spanning tree for $\Delta$, and $\Psi_T$ a 1-combing constructed as in the proof of Proposition~\ref{Nice 1-combings exist}.  Given $x \in \partial \Delta$, there exists a sequence $1 \ge t_1 > \dots > t_{m+1} > t_{m+2} = 0$ with the following properties.  Let $x_j = \Psi_T(x,t_j)$ for all $j \in [m+2]$.  Then for each $j \in [m]$ there is a 1-cell $e_j \in E(\Delta \setminus T)$ with $x_j \in \interior{e_j}$.  Furthermore, $\Psi_T(x,[0, t_{m+1}])$ is the simple path from $*$ to $x_{m+1}$ in $T_{i+1}$, $\Psi_T(x,[t_1,1]) = \{x\}$, and for all $j \in [m]$, $t_j = \text{min}\{t \in [0,1] | \Psi_T(x,t) = x_j \}$, and $\Psi_T(x,[t_{j+1},t_j]) = \phi_{\sigma_j}^{-1}(h(\phi_{\sigma_j}(x_j),[0,1]))$, where $\sigma_j := \theta_T(e_j)$.
\end{Lemma}

\begin{proof}
Let $x \in \partial \Delta$.  For each $i \in \{0, 1, \dots, n+1\}$, let $s_i = 1 - \frac{i}{n+1}$ and let $y_i = \Psi_T(x, s_i)$.  Also, for $i \in \{0, 1, \dots, n-1\}$, either $y_i \in \interior{e_{i+1}}$, in which case $\Psi_T(x,[s_{i+1},s_i]) = H_{i+1}(y_i,[0,1])$, or $y_i \not \in \interior{e_{i+1}}$, in which case $\Psi_T(x,[s_{i+1},s_i]) = \{y_i\} = \{y_{i+1}\}$.  Let $i_1 < i_2 < \dots < i_m$ be the collection of indices such that $y_{i_j} \in \interior{e_{i_j+1}}$, and let $i_{m+1} = n$.  Then for all $j \in [m+1]$, let $t_j = s_{i_j}$ and $x_j = y_{i_j}$ Then we know that $\Psi_T(x,[0,t_{m+1}]) =: \gamma_x$, the unique simple path from $*$ to $x_{m+1}$ in $T$.  Then for all $i \in \{0, \dots, i_1-1\}$, since $\Psi_T(x,[s_{i+1},s_i]) = \{y_i\} = \{y_{i+1}\}$, we have that $x = y_i = x_1$.  Hence, $\Psi_T(x, [t_1,1]) = \{x\}$.  Similarly, for all $j \in [m]$ and $i \in \{i_j+1, \dots, i_{j+1}-1\}$, $y_{i_j+1} = y_i = x_{j+1}$, so $\Psi_T(x, [t_{j+1},s_{i_j+1}]) = \{x_{j+1}\}$.  Now for $j \in [m]$, $$H_{i_j+1}(x_j,1) = \Psi_T(x,s_{i_j+1}) = y_{i_j+1} = x_{j+1},$$ and therefore 

\begin{align*}
\Psi_T(x,[t_{j+1},t_j]) &= \Psi_T(x,[t_{j+1},s_{i_j+1}]) \cup \Psi_T(x,[s_{i_j+1},t_j]) \\
&= \{x_{j+1}\} \cup H_{i_j+1}(x_j,[0,1]) \\
&= H_{i_j+1}(x_j,[0,1]) \\
&= \phi_{\sigma_{i_j+1}}^{-1}(h(\phi_{\sigma_{i_j+1}}(x_j),[0,1])).
\end{align*}

Note that by definition $\sigma_{i_j+1} = \theta_T(e_{i_j+1})$, and $e_{i_j+1}$ is the 1-cell containing $x_j$, so this matches the statement of the lemma.  

We have only left to show that $t_j = \text{min}\{t \in [0,1] | \Psi_T(x,t) = x_j \}$.  Note that the above implies that $$\Psi_T(x,[0,1]) = \gamma_x \cup \bigcup_{j \in [m]} H_{i_j+1}(x_j,[0,1]).$$  Note also that for each $j, k \in [m]$ with $k > j$, $x_j \neq x_k$, since $x_j \in \interior{e_{i_j+1}}$ and $\interior{e_{i_j+1}}$ is not in $\Delta_{i_j+1}$, whereas $x_k \in \interior{e_{i_k+1}}$, and $e_{i_k+1} \subseteq \Delta_{i_k} \subseteq \Delta_{i_j+1}$ since $i_k \ge i_j+1$.  Furthermore, for all $j, k \in [m]$, $x_j \not \in \Psi_T(x, (s_{i_k+1}, t_k)) = H_{i_k+1}(x_k,(0,1))$, since $H_{i_k+1}(x_k,(0,1)) \cap \Delta^{(1)} = \emptyset$.  Also, for all $j \in [m]$, $x_j \not \in \gamma_x$, since $\gamma_x \subseteq T$ but $x_j \in \interior{e_{i_j+1}}$, and $e_{i_j+1} \in E(\Delta \setminus T)$.  This shows that $x_j \not \in \Psi_T(x, [0,t_j))$.  In other words, $t_j = \text{min}\{t \in [0,1] | \Psi_T(x,t) = x_j\}$.
\end{proof}

\section{Constructing Simply, Geodesically Bounded Van Kampen Diagrams of Small Diameter}

Our main strategy for proving Theorem~\ref{main} will be to take a van Kampen diagram with a geodesic spanning tree and start replacing the bodies of icicles whose diameter is too large with diagrams that have smaller intrinsic diameter.  This will cause 1-combings respecting the tree to become more and more tame.  In order for this replacement to happen while preserving the necessary structure, the diagrams with which we replace the bodies of icicles must have two important properties.

\begin{Def} A 2-complex is \emph{simply bounded} if it is bounded by a simple circuit.
\end{Def}

Note that the bodies of icicles are always simply bounded by definition, so the replacement diagrams must also be simply bounded in order for them to be bodies of icicles in the same tree.

\begin{Def} Let $\Delta$ be a van Kampen diagram with basepoint $*$.  Then we say that $\Delta$ is \emph{geodesically bounded} if $d_\Delta(*,x) = d_{\partial \Delta}(*,x)$ for all $x \in \Delta^{(0)} \cap \partial \Delta$.  
\end{Def}

Note that if we consider the body of an icicle of a tree of geodesics to be a van Kampen diagram with basepoint at the intersection of the tail and body of the icicle, it is geodesically bounded.  The replacement diagrams must share this property in order to preserve distances in the resulting diagram.

The purpose of this section is to show that, given a van Kampen diagram that could be the body of an icicle---in other words, a diagram that is simply bounded---there is always another van Kampen diagram for the same word that is both simply bounded and geodesically bounded and has relatively small intrinsic diameter.  In order to say what "relatively small" means, we need a few more definitions.

Recall that $\mathcal{P} = \langle A | R \rangle$ is a finite presentation.  Given $w \in (A \cup A^{-1})^*$, let SB$_w$ be the set of all simply bounded van Kampen diagrams for $w$.  Define the simply bounded diameter of $w$ by
$$\IDiam_{\text{sb}}(w) = \text{inf}\{\IDiam (D) | D \in \text{SB}_w \}$$
and choose $D_w$ to be a van Kampen diagram for $w$ that attains this infimum if such a diagram exists.  Note that the infimum is attained if and only if $\text{SB}_w \neq \emptyset$.  Let
$$M_\mathcal{P} = \text{max}( \{0\} \cup \{ \IDiam_{\text{sb}}(w) | w \in (A \cup A^{-1})^* \text{ such that } \ell(w) \leq 4 \text{ and } \text{SB}_w \neq \emptyset \} ).$$

\begin{Prop} \label{fat diagrams} Let $\mathcal{P} = \langle A | R \rangle$ be a finite presentation for a group $G$ such that no generator is equal to the identity.  Let $w \in (A \cup A^{-1})^*$ with $w =_G 1$.  Suppose that there exists a simply bounded van Kampen diagram for $w$.  Then there exists a simply and geodesically bounded van Kampen diagram $\Delta$ for $w$ such that $$\IDiam(\Delta) \leq \text{max} \left(\IDiam(\ell(w)), \left\lfloor \frac{\ell(w)}{2} \right\rfloor + M_\mathcal{P} \right).$$
\end{Prop}

\begin{proof}
Case 1: $\ell(w) \leq 2$.  Since there exists a simply bounded van Kampen diagram for $w$, $\text{SB}_w \neq \emptyset$.  So let $\Delta = D_w$.  Then $\IDiam(\Delta) \leq M_\mathcal{P} \leq \text{max} (\IDiam(\ell(w)), \left\lfloor \frac{\ell(w)}{2} \right\rfloor + M_\mathcal{P})$.  Also, if there exists a vertex $x \neq *$ on the boundary, then $d_{\Delta}(*,x) = 1 = d_{\partial \Delta}(*,x)$, so $\Delta$ is geodesically bounded.  Thus, $\Delta$ satisfies the desired conditions.

\begin{figure}[h!]
    \includegraphics[width=1.5in]{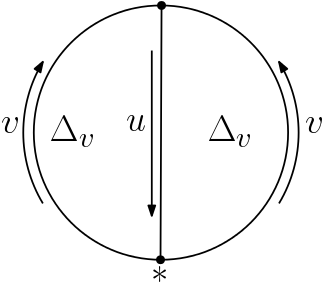}
    \caption{Showing $\text{SB}_{vv^{-1}} \neq \emptyset$ in Case 2}
    \label{fig:D_v def}
\end{figure}

Case 2: $\ell(w) > 2$.  I will first show that for any subword $v$ of $w$ with $\ell(v) = 2$, $\text{SB}_{vv^{-1}} \neq \emptyset$.  We will therefore be free to use the diagrams $D_{vv^{-1}}$ in order to construct $\Delta$ from the statement of the proposition.  Construct a diagram $\Delta_v$ by taking a simply bounded van Kampen diagram for $w$ and moving the basepoint so that $\Delta_v$ is a van Kampen diagram for $vu$, for some nonempty word $u$.  Then since $vu$ labels a simple circuit, $u$ labels a simple path on the boundary.  Therefore, taking two copies of $\Delta_v$ and gluing them along their respective paths labeled by $u$ results in a van Kampen diagram for $vv^{-1}$, shown in Figure~\ref{fig:D_v def}.  Note that the boundary of this van Kampen diagram is a simple circuit since $u$ is nonempty.  Therefore, $\text{SB}_{vv^{-1}}$ contains this diagram.

Let $\Delta_0$ be a van Kampen diagram for $w$ of minimum intrinsic diameter.  Let $* = p_0, p_1, \dots, p_n = *$ be the sequence of vertices traversed in order by the path labeled by $w$ along the boundary of $\Delta_0$.  Let $I_0 = \{i_1, \dots, i_m\} \subseteq [n-1]$ be the set of indices such that the vertices they index appeared earlier on the boundary; that is, $$I_0 = \{ i \in [n-1] | \text{there exists } k < i \text{ such that } p_i = p_k \}.$$  The vertices at these indices are cut vertices of $\Delta_0$, and therefore obstructions to making it simply bounded.  If $I_0 = \emptyset$, then $\Delta_0$ is simply bounded and $m = 0$.  Otherwise, we will construct a sequence of van Kampen diagrams for $w$ $\Delta_0, \Delta_1, \dots, \Delta_m$ such that $\Delta_m$ is simply bounded and $\IDiam(\Delta_m) \leq \text{max} \left(\IDiam(\ell(w)), \left\lfloor \frac{\ell(w)}{2} \right\rfloor + M_\mathcal{P}\right)$.  Our strategy will be to add a bit of ``padding" to the diagram at a cut vertex at each step, reducing the number of indices corresponding to a cut vertex by 1.

\begin{figure}[h!]
  \centering
  \begin{subfigure}[b]{.55\linewidth}
    \includegraphics[width=\linewidth]{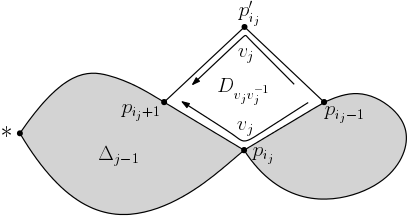}
    \caption{Defining $\Delta_j$ for $0 \leq j < m$}
  \end{subfigure}
  \hspace{.05\linewidth}
  \begin{subfigure}[b]{.35\linewidth}
    \includegraphics[width=\linewidth]{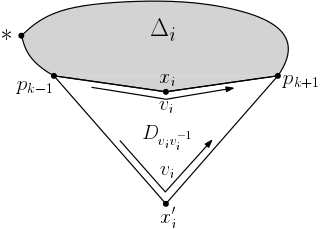}
    \caption{Defining $\Delta_{i+1}$ for $m \leq i < l$}
  \end{subfigure}
  \caption{Defining $\Delta_1, \dots \Delta_l$ in Case 2 of Proposition~\ref{fat diagrams}}
  \label{fig:Case 2 diagrams}
\end{figure}

For $j \in [m]$, assume by induction that we have constructed a van Kampen diagram $\Delta_{j-1}$ for $w$ such that $* = p_0, p_1, \dots, p_n = *$ is the sequence of vertices traversed in order by the path labeled by $w$ along the boundary of $\Delta_{j-1}$, that $p_0, p_1,\dots,p_{i_j-1}$ are all distinct, and that $\IDiam(\Delta_{j-1}) \leq \text{max} \left(\IDiam(\ell(w)), \left\lfloor \frac{\ell(w)}{2} \right\rfloor + M_\mathcal{P}\right)$.  Let $$I_{j-1} = \{ i \in [n-1] | \text{there exists } k < i \text{ such that } p_i = p_k \}$$ and assume by induction that $I_{j-1} = \{i_j, \dots, i_m\}$.  Then construct $\Delta_j$ from $\Delta_{j-1}$ in the following way.  Let $\gamma_j$ be the path along the boundary of $\Delta_{j-1}$ from $p_{i_j-1}$ to $p_{i_j+1}$ and let $v_j$ be the subword of $w$ that labels $\gamma_j$.  Note that $\ell(v_j) = 2$, so we have shown above that $\text{SB}_{v_jv_j^{-1}} \neq \emptyset$ and therefore $D_{v_jv_j^{-1}}$ is defined.  Let $\eta_j$ be a path in $D_{v_jv_j^{-1}}$ along the boundary starting at the basepoint labeled by $v_j$. Since $D_{v_jv_j^{-1}}$ is simply bounded, $\eta_j$ is a simple path, and I claim that $\gamma_j$ is as well.  Since no generator is equal to the identity, we know that $p_{i_j-1} \neq p_{i_j} \neq p_{i_j+1}$, so we only need to show that $p_{i_j-1} \neq p_{i_j+1}$.  Since $i_j \in I_{j-1}$, there is some $k < i_j$ such that $p_k = p_{i_j}$.  Using Lemma~\ref{gamma_j simple path}, proven below, this implies that $p_{i_j-1} \neq p_{i_j+1}$, making $\gamma_j$ a simple path.

Now construct $\Delta_j$ from $\Delta_{j-1}$ and $D_{v_jv_j^{-1}}$ by gluing $\eta_j$ along $\gamma_j$, as in Figure~\ref{fig:Case 2 diagrams}A, and let $q_j: \Delta_{j-1} \sqcup D_{v_jv_j^{-1}} \to \Delta_j$ be the corresponding quotient map.  Because $\gamma_j$ and $\eta_j$ are both simple paths along the boundaries of their respective diagrams, $\Delta_j$ is planar and simply connected.  Let $p_k' = q_j(p_k)$ for $k \neq i_j$, and let $p_{i_j}'$ be the vertex on the boundary of $q_j(D_{v_jv_j^{-1}})$ that is not in $q_j(\eta_j)$.  Then the path along the boundary of $\Delta_j$ from $p_{i_j-1}'$ to $p_{i_j+1}'$ through $p_{i_j}'$ is the image under $q_j$ of a path of length 2 along the boundary of $D_{v_jv_j^{-1}}$ starting at the basepoint, and is therefore labeled by $v_j$.  So $\Delta_j$ is a van Kampen diagram for $w$, and $* = p_0', p_1', \dots, p_n' = *$ is the sequence of vertices traversed in order by the path labeled by $w$ along the boundary of $\Delta_j$.  Additionally, $p_{i_j}' \neq p_k'$ for all $k < i_j$, so $p_0',p_1',\dots,p_{i_j}'$ are all distinct.  Furthermore, since $d_{\Delta_{j-1}}(*,p_{i_j-1}) \leq \left\lfloor \frac{\ell(w)}{2} \right\rfloor$ and $q_j$ identifies the basepoint of $D_{v_jv_j^{-1}}$ with $p_{i_j-1}$, we know that for any $y \in q_j(D_{v_jv_j^{-1}})^{(0)}$, 
$$d_{\Delta_j}(*,y) \leq d_{\Delta_{j-1}}(*,p_{i_j-1}) + \IDiam(D_{v_jv_j^{-1}}) \leq \left\lfloor \frac{\ell(w)}{2} \right\rfloor + M_\mathcal{P}.$$
Hence, 

\begin{align*}
\IDiam(\Delta_j) &\leq \text{max} \left(\IDiam(\Delta_{j-1}), \left\lfloor \frac{\ell(w)}{2} \right\rfloor + M_\mathcal{P}\right) \\
&\leq \text{max} \left(\IDiam(\ell(w)), \left\lfloor \frac{\ell(w)}{2} \right\rfloor + M_\mathcal{P}\right).
\end{align*}

Let $$I_j = \{ i \in [n-1] | \text{there exists } k < i \text{ such that } p_i' = p_k' \}.$$  Since $q_j|_{\Delta_{j-1}}$ is a bijection, for $0 \leq k < i < n$ with $i,k \neq i_j$, we have that $p_i' = p_k'$ if and only if $p_i = p_k$.  Also, $p_{i_j}'$ is only the image under $q_j$ of the one boundary vertex of $D_{v_jv_j^{-1}}$ that is not on $\eta_j$, and not the image of any $p_i$.  Therefore, $p_{i_j}' \neq p_i'$ for any $i \neq i_j$.  As a result, $I_j = I_{j-1} \setminus \{i_j\} = \{i_{j+1}, \dots, i_m\}$.  This completes the induction.

So $\Delta_{m}$ is a van Kampen diagram for $w$, and $I_m = \emptyset$.  Since $\ell(w)>2$, this implies that the boundary of $\Delta_{m}$ is a simple circuit.  Furthermore, $$\IDiam(\Delta_m) \leq \text{max} \left(\IDiam(\ell(w)), \left\lfloor \frac{\ell(w)}{2} \right\rfloor + M_\mathcal{P}\right).$$  However, $\Delta_m$ may not be geodesically bounded; we may need to continue with a similar process in order to acquire a geodesically bounded diagram with these properties.

Given any van Kampen diagram $\Delta$, define the \emph{cumulative boundary diameter} of $\Delta$ to be
$$\bdiam(\Delta) := \sum_{x \in \Delta^{(0)} \cap \partial \Delta} d_{\Delta}(*,x).$$
Note that the cumulative boundary diameter is bounded above with respect to the length of the boundary circuit of $\Delta$.  If $\Delta$ is a van Kampen diagram for a word $w$, then for each $x \in \Delta^{(0)} \cap \partial \Delta$ we have $d_\Delta(*,x) \leq d_{\partial \Delta}(*,x) \leq \frac{1}{2}\ell(w)$, and therefore $\bdiam(\Delta) \leq \frac{1}{2}\ell(w)^2$.

We will construct a sequence of van Kampen diagrams $\Delta_m, \Delta_{m+1}, \dots, \Delta_l$ for $w$ such that if $m \leq i \leq l$ then $\IDiam(\Delta_i) \leq \text{max} \left(\IDiam(\ell(w)), \left\lfloor \frac{\ell(w)}{2} \right\rfloor + M_\mathcal{P} \right)$, $\Delta_i$ is simply bounded, and, for $i > m$, $\bdiam(\Delta_i) > \bdiam(\Delta_{i-1})$.  The fact that the cumulative boundary diameter keeps increasing implies that at some point, each point on the boundary of the diagram will reach its maximum possible distance from the basepoint in a van Kampen diagram for $w$, resulting in a geodesically bounded diagram.  Assuming that these properties hold for $i \ge m$, we will construct $\Delta_{i+1}$ as follows.  Let $$X_i = \{x \in \Delta_i^{(0)} \cap \partial \Delta_i | d_{\Delta_i}(*,x) < d_{\partial \Delta_i}(*,x) \},$$ the set of vertices $x$ along the boundary of $\Delta_i$ such that there is no $\Delta_i$-geodesic from $*$ to $x$ in $\partial \Delta_i$.  If $X_i = \emptyset$, then $\Delta_i$ is geodesically bounded, so we will let $l := i$ and the construction of the sequence ends.  Otherwise, let $m_i = \text{min}(\{d_{\Delta_i}(*,x) | x \in X_i\})$ and let $x_i \in X_i$ be a vertex that attains this minimum.  Let $* = p_0, p_1, \dots, p_{n-1}, p_n = *$ be the sequence of vertices traversed in order by the path labeled by $w$ along the boundary of $\Delta_{i}$, and let $k \in [n]$ such that $p_k = x_i$.  Note that by the definition of $X_i$, $* \not \in X_i$, so $0 < k < n$.  Let $\gamma_i$ be the path along the boundary from $p_{k-1}$ to $x_i$ to $p_{k+1}$, and let $v_i$ be the subword of $w$ of length 2 that labels $\gamma_i$.  Let $\eta_i$ be a path in $D_{v_iv_i^{-1}}$ along the boundary starting at the basepoint labeled by $v_i$.  Note that $\gamma_i$ and $\eta_i$ and both simple paths, since they are paths of length 2 along the boundaries of simply bounded diagrams of a word $w$ with $\ell(w)>2$.

Construct $\Delta_{i+1}$ from $\Delta_i$ and $D_{v_iv_i^{-1}}$ by gluing $\eta_i$ along $\gamma_i$, as in Figure~\ref{fig:Case 2 diagrams}B, and let $q_i: \Delta_i \sqcup D_{v_iv_i^{-1}} \to \Delta_{i+1}$ be the corresponding quotient map.  Because $\gamma_i$ and $\eta_i$ are both simple paths along the boundaries of their respective diagrams, $\Delta_{i+1}$ is planar and simply connected.  Let $p_j' = q_i(p_j)$ for $j = 0, \dots, n$ and let $x_i'$ be the vertex on the boundary of $q_i(D_{v_iv_i^{-1}})$ that is not in $q_i(\eta_i)$.  Note that, by the exact same reasoning as in the construction of $\Delta_1$ through $\Delta_{m}$, $\Delta_{i+1}$ is a van Kampen diagram for $w$, $\Delta_{i+1}$ is simply bounded, and $\IDiam(\Delta_{i+1}) \leq \text{max} \left(\IDiam(\ell(w)), \left\lfloor \frac{\ell(w)}{2} \right\rfloor + M_\mathcal{P}\right)$.

We have only to show that $\bdiam(\Delta_{i+1}) > \bdiam(\Delta_i)$.  First, I claim that for any $y \in \Delta_i^{(0)}$, we have that $d_{\Delta_{i+1}}(*,q_i(y)) = d_{\Delta_i}(*,y)$; gluing $D_{v_iv_i^{-1}}$ to $\Delta_i$ has not changed the distance from $*$ to the image of $y$.  It is sufficient to show that $d_{\Delta_{i+1}}(*,q_i(y)) = d_{\Delta_i}(*,y)$ for $y \in \{x_i, p_{k-1}, p_{k+1}\}$.  Now since $d_{\Delta_i}(*,x_i) = m_i$, $d_{\Delta_i}(*,p_{k-1}) \in \{m_i-1,m_i,m_i+1\}$.  Suppose that $d_{\Delta_i}(*,p_{k-1}) = m_i-1$.  Then by the definition of $m_i$, $d_{\Delta_i}(*,p_{k-1}) = d_{\partial \Delta_i}(*,p_{k-1})$.  So there is a $\Delta_i$-geodesic along the boundary of $\Delta_i$ from $*$ to $p_{k-1}$.  Since $x_i$ is adjacent to $p_{k-1}$ via a 1-cell on the boundary of $\Delta_i$ and $d_{\Delta_i}(*,p_{k-1}) + 1 = d_{\Delta_i}(*,x_i)$, this geodesic extends to a geodesic from $*$ to $x_i$ along the boundary of $\Delta_i$. This contradicts the fact that $x_i \in X_i$.  Thus, $d_{\Delta_i}(*,p_{k-1}) \in \{m_i, m_i+1\}$.  The same argument shows that $d_{\Delta_i}(*,p_{k+1}) \in \{m_i, m_i+1\}$.  So the distances from $*$ to $x_i$, $p_{k-1}$, and $p_{k+1}$ in $\Delta_i$ all differ from each other by at most 1.  Now let $y \in \{x_i, p_{k-1}, p_{k+1}\}$ and let $\gamma$ be a $\Delta_{i+1}$-geodesic from $*$ to $q_i(y)$.  Let $z \in \{x_i, p_{k-1}, p_{k+1}\}$ such that $q_i(z)$ is the first point at which $\gamma$ enters $q_i(D_{v_iv_i^{-1}})$.  If $z = y$, $\gamma$ is contained in $q_i(\Delta_i)$, so $d_{\Delta_{i+1}}(*,q_i(y)) = d_{\Delta_{i}}(*,y)$.  Otherwise, $z \neq y$.  Then since $\gamma$ is a $\Delta_{i+1}$-geodesic which, up until $q_i(z)$, is contained in $q_i(\Delta_i)$, $d_{\Delta_{i+1}}(*,q_i(z)) = d_{\Delta_i}(*,z) \ge m_i$.  Since distance increases along geodesics, this implies that $d_{\Delta_{i+1}}(*,q_i(y)) \ge m_i+1 \ge d_{\Delta_{i}}(*,y)$.  So $d_{\Delta_{i+1}}(*,q_i(y)) = d_{\Delta_{i}}(*,y)$.

Since $\partial \Delta_{i+1} \setminus q_i(\partial \Delta_i) = \{x_i'\}$ and $q_i(\partial \Delta_i) \setminus \partial \Delta_{i+1} = \{q_i(x_i)\}$, and $d_{\Delta_{i+1}}(*,q_i(y)) = d_{\Delta_{i}}(*,y)$ for all $y \in \partial \Delta_i$, we have that $$\bdiam(\Delta_{i+1}) - \bdiam(\Delta_i) = d_{\Delta_{i+1}}(*,x_i') - d_{\Delta_i}(*,x_i).$$  Now any path from $*$ to $x_i'$ in $\Delta_{i+1}$ passes through one of $q_i(x_i)$, $q_i(p_{k-1})$, or $q_i(p_{k+1})$.  Hence, $d_{\Delta_{i+1}}(*,x_i') > m_i = d_{\Delta_i}(*,x_i)$.  Thus, $\bdiam(\Delta_{i+1}) - \bdiam(\Delta_i) > 0$.  Also, since $\Delta_i$ is a van Kampen diagram for $w$ for all $i$, we showed above that $\bdiam(\Delta_i) \leq \frac{1}{2}\ell(w)^2$.

Since $\bdiam(\Delta_i)$ is increasing and bounded above, the construction of this sequence must end with some $\Delta_l$.  Then $X_l = \emptyset$, which implies that, for all $x \in \Delta_l^{(0)} \cap \partial \Delta_l$, $d_{\Delta_l}(*,x) = d_{\partial \Delta_l}(*,x)$.  So $\Delta_l$ is geodesically bounded.

\end{proof}

We used the following technical lemma to show that $\gamma_j$ was a simple path when constructing $\Delta_{j+1}$ for $j+1 \in [m]$.  We will now tie up this loose end.

\begin{Lemma} \label{gamma_j simple path} Let $\Delta$ be a van Kampen diagram for a word $w$ in a finite presentation in which no generator is equal to the identity.  Let $p_1, p_2, \dots, p_n = p_1$ be the sequence of vertices traversed in order by the path labeled by $w$ along the boundary of $\Delta$.  Suppose there exist $i,j \in [n]$ with $i < j$ such that $p_i = p_j$ and $p_k \neq p_l$ whenever $i \leq k < l < j$.  Then $p_{j-1} \neq p_{j+1}$.
\end{Lemma}

\begin{proof}

Note first the following general facts about the boundary circuit of a van Kampen diagram.  Recall that each 1-cell in a van Kampen diagram is associated to two directed edges, one going each direction.  The boundary of each 2-cell is a circuit which can be directed either clockwise or counterclockwise.  It is possible to choose a directed boundary circuit for each 2-cell (going either clockwise or counterclockwise) such that each directed edge in the diagram appears exactly once in either the directed boundary circuit of one of the 2-cells or the boundary circuit for the diagram, but not both; see, for example, \cite[p. 236]{L&S}.   So each directed edge corresponding to a 1-cell on the boundary of a van Kampen diagram appears at most once in the boundary circuit of the diagram.  If both directed edges appear in the boundary circuit, the corresponding 1-cell must be incident to the complement of the diagram on both sides, since otherwise one of the directed edges would appear in the directed boundary circuit of a 2-cell. If only one of the two directed edges appears in the boundary circuit, then the other appears in the directed boundary circuit of a 2-cell, so the corresponding 1-cell is incident to a 2-cell of the diagram on one side and incident to the complement of the diagram on the other.

If $e$ is a directed edge, let $\bar{e}$ denote the other edge corresponding to the same 1-cell, going in the opposite direction.

Let $i,j \in [n]$ with $i < j$ such that $p_i = p_j$ and $p_k \neq p_l$ whenever $i \leq k < l < j$.  Suppose by way of contradiction that $p_{j-1} = p_{j+1}$.  Let $e$ be the directed edge in the boundary circuit from $p_{j-1}$ to $p_{j}$, and let $f$ be the directed edge in the boundary circuit from $p_{j}$ to $p_{j+1}$.  Let $\gamma = ef$, the directed path in the boundary circuit from $p_{j-1}$ to $p_{j+1}$.  Since no generator is equal to the identity, we know that $p_j \neq p_{j-1} = p_{j+1}$.  So either $\gamma$ is a simple circuit or $e = \bar{f}$.  Let $\eta$ be the directed path in the boundary circuit from the $i$th vertex $p_i$ to the $j$th vertex $p_j$.  Note that $e$ is the last edge of $\eta$ and $\eta$ is a circuit that does not repeat any vertices, by the assumption that $p_k \neq p_l$ whenever $i \leq k < l < j$.  Therefore, either $\eta$ is a simple circuit or $\eta = \bar{e}e$.

Suppose first that $e = \bar{f}$.  In this case, both $e$ and $\bar{e}$ appear in $\gamma$, and therefore both appear in the boundary circuit, so $e$ is incident to $\R^2 \setminus \Delta$ on both sides.  If $\eta$ is a simple circuit and $e$ appears in $\eta$, $e$ is incident to the interior of $\eta$.  This would imply that the interior of $\eta$, a simple circuit in $\Delta$, contains a point in $\R^2 \setminus \Delta$, contradicting the fact that $\Delta$ is simply-connected.  If $\eta = \bar{e}e$, then $\bar{e}$ would appear twice in the boundary circuit: once on $\eta$ and once on $\gamma$.  This is a contradiction, so $e \neq \bar{f}$.

\begin{figure}[h!]
    \includegraphics[width=5in]{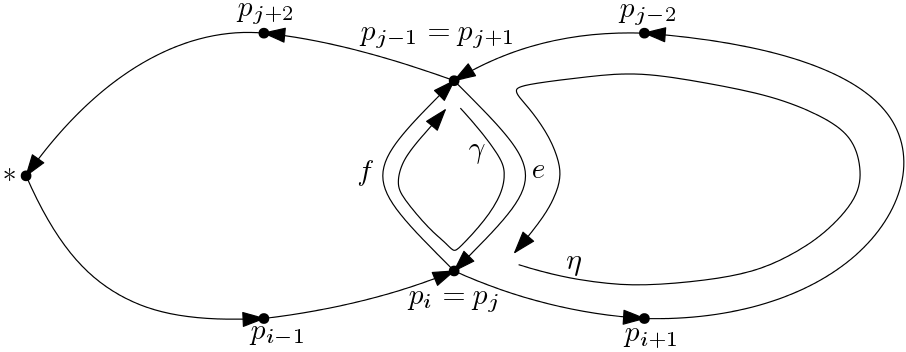}
    \caption{One possible drawing of $\Delta$ in the case where $\gamma$ and $\eta$ are both simple circuits.  The arrows indicate the direction of the boundary circuit as well as the orientations of $e$ and $f$.}
    \label{fig:If both circuits}
\end{figure}

So suppose instead that $\gamma$ is a simple circuit.  In this case, $e$ and $f$ are both incident to a 2-cell on one side, and thus neither $\bar{e}$ nor $\bar{f}$ appear in the boundary circuit.  Therefore, $\eta \neq \bar{e}e$.  If instead $\eta$ is a simple circuit, we are in the case shown in Figure~\ref{fig:If both circuits}.  Then $\interior{f}$ must either be in the interior or exterior of $\eta$; $f$ cannot be a part of $\eta$, since it appears only once in the boundary circuit in $\gamma$, and we know that $\bar{f}$ does not appear on the boundary circuit.  $f$ cannot be in the interior of $\eta$, since it is incident to the complement of $\Delta$, so $\interior{f}$ is in the exterior of $\eta$.  In this case, either $\eta \setminus e$ is in the interior of $\gamma$ or neither contains the other.  Now each edge of $\eta$ is incident to the complement of $\Delta$, so $\eta \setminus e$ cannot be in the interior of $\gamma$; then its interior would contain a point in the complement of $\Delta$.  So the interiors of $\gamma$ and $\eta$ do not intersect.  But since $e$ is on both $\gamma$ and $\eta$, it is incident to the interiors of both.  Since these interiors do not intersect, $e$ must be incident to the interior of $\gamma$ on one side and the interior of $\eta$ on the other.  This contradicts the fact that $e$ is incident to the complement of $\Delta$.  This final contradiction implies that $p_{j-1} \neq p_{j+1}$.

\end{proof}

\section{Main Theorem}

In this section, we will finish the proof of Theorem~\ref{main}.  In order to simplify the proof, we will first define a variant of intrinsic tame filling functions that only relies on distance to the 1-skeleton of a van Kampen diagram and prove that such functions are equivalent to tame filling functions.  This will allow us to ignore distance to 2-cells in the remainder of the proof.

\begin{Def} Let $f: \N[\frac{1}{4}] \to \N[\frac{1}{4}]$ be a function, $C$ a 2-complex with basepoint $*$ and $D$ a subcomplex of $C^{(1)}$.  A 1-combing $\Psi$ of the pair $(C,D)$ is \emph{graph $f$-tame} if for all $n \in \N[\frac{1}{4}]$ and for all $x \in D$ and $s,t \in [0,1]$ such that $s \leq t$ and $\Psi(x,s),\Psi(x,t) \in C^{(1)}$, if $d_{C}(*,\Psi(x,t)) \leq n$, then $d_{C}(*,\Psi(x,s)) \leq f(n)$.
\end{Def}

\begin{Def} Let $f: \N[\frac{1}{4}] \to \N[\frac{1}{4}]$ be non-decreasing.  $f$ is an \emph{intrinsic graph tame filling function} for $\langle A | R \rangle$ if, for all $w \in (A \cup A^{-1})^*$ with $w =_G 1$, there is a van Kampen diagram $\Delta_w$ for $w$ with basepoint $*$ and a 1-combing $\Psi_w$ of $(\Delta_w, \partial \Delta_w)$ based at $*$ such that $\Psi_w$ is graph $f$-tame.
\end{Def}

\begin{Lemma} \label{graph tame} Every intrinsic graph tame filling function for $\langle A | R \rangle$ is equivalent to an intrinsic tame filling function for $\langle A | R \rangle$.  Conversely, every intrinsic tame filling function is itself an intrinsic graph tame filling function.
\end{Lemma}

\begin{proof}
We will first show that all intrinsic tame filling functions are intrinsic graph tame filling functions.  Let $h$ be an intrinsic tame filling function and let $w \in (A \cup A^{-1})^*$ with $w =_G 1$.  Let $\Delta_w$ be a van Kampen diagram for $w$ and $\Psi_w$ an $h$-tame 1-combing of $\Delta_w$.  Since $\Psi_w$ is $h$-tame, we have that for all $n \in \N[\frac{1}{4}]$, $p \in \partial \Delta_w$, and $s,t \in [0,1]$ such that $s \leq t$, if $d_{\Delta_w}(*,\Psi_w(p,t)) \leq n$, then $d_{\Delta_w}(*,\Psi_w(p,s)) \leq h(n)$.  Since this is true for all such $p$, $s$, and $t$, it is also true whenever $\Psi(p,s),\Psi(p,t) \in \Delta_w^{(1)}$.  Therefore, $\Psi_w$ is graph $h$-tame.  So $h$ is also an intrinsic graph tame filling function.

For the other direction, let $f$ be an intrinsic graph tame filling function.  Let $\rho = \text{max}\{\ell(r) | r \in R\}$.  Define $g: \N[\frac{1}{4}] \to \N[\frac{1}{4}]$ by $g(n) = f \left( n + \frac{3}{4} \right) + \frac{\rho}{2} - \frac{1}{4}$.  Note that $g$ is equivalent to $f$.  We will show that $g$ is an intrinsic tame filling function for $\langle A | R \rangle$.

Let $w \in (A \cup A^{-1})^*$ with $w =_G 1$.  Let $\Delta_w$ be a van Kampen diagram for $w$ and $\Psi_w$ a graph $f$-tame 1-combing of $\Delta_w$.  Let $n \in \N[\frac{1}{4}]$, $p \in \partial \Delta_w$, and $s,t \in [0,1]$ such that $s \leq t$ and $d_{\Delta_w}(*,\Psi_w(p,t)) \leq n$. We will show that $\Psi_w$ is $g$-tame by showing that $d_{\Delta_w}(*,\Psi_w(p,s)) \leq g(n)$.  See Figure~\ref{fig:Graph Tame}.

\begin{figure}[h!]
  \includegraphics[width=\linewidth]{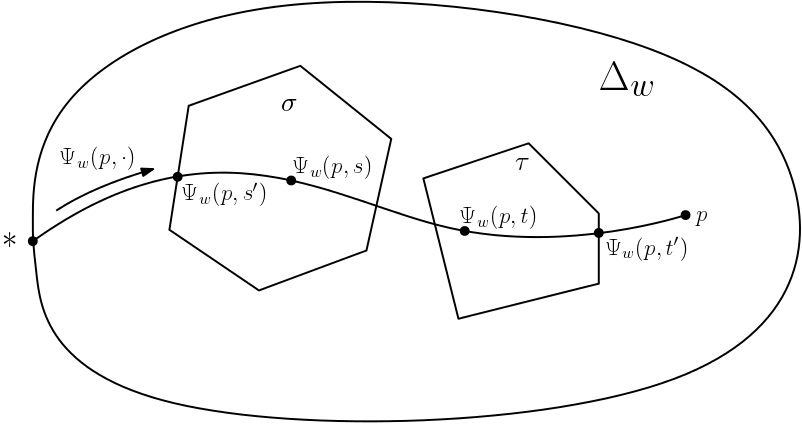}
  \caption{Showing $\Psi_w$ is $g$-tame.}
  \label{fig:Graph Tame}
\end{figure}

Choose $t' \in [0,1]$ such that $t' \ge t$ in the following way.  If $\Psi_w(p,t) \in \Delta_w^{(1)}$, let $t' = t$.  Otherwise, let $\tau$ be the 2-cell with $\Psi_w(p,t)$ in its interior.  Since $p \in \Delta_w^{(1)}$, and $\Psi_w(p, \cdot)$ is continuous, there is some $t' > t$ such that $\Psi_w(p,t') \in \partial \tau$.  Then by the definition of the coarse distance function, $d_{\Delta_w}(*,\Psi_w(p,t')) \leq d_{\Delta_w}(*,\Psi_w(p,t)) + \frac{3}{4} \leq n + \frac{3}{4}$.

Now choose $s' \in [0,1]$ such that $s' \leq s$ in the following way.  If $\Psi_w(p,s) \in \Delta_w^{(1)}$, let $s' = s$.  Otherwise, let $\sigma$ be the 2-cell with $\Psi_w(p,s)$ in its interior.  Since $* \in \Delta_w^{(1)}$, and $\Psi_w(p, \cdot)$ is continuous, there is some $s' < s$ such that $\Psi_w(p,s') \in \partial \sigma$.  Now there are at most $\rho$ 1-cells and $\rho$ vertices in $\partial \sigma$.  Suppose $q$ is a point in $\partial \sigma$ with maximum coarse distance from $*$.  Then $d_{\Delta_w}(*, q) \ge d_{\Delta_w}(*, \Psi_w(p,s)) + \frac{1}{4}$.  Now, there is a path in $\partial \sigma$ from $q$ to $\Psi_w(p,s')$ that passes through at most $\rho$ 1-cells not containing $q$ and $\rho$ vertices not equal to $q$.  At each step from a vertex to the interior of a 1-cell or the interior of a 1-cell to a vertex, the coarse distance changes by $\frac{1}{2}$.  Since there are at most $\rho$ such steps from $q$ to $\Psi_w(p,s')$, 
$$d_{\Delta_w}(*, \Psi_w(p,s')) \ge d_{\Delta_w}(*, q) - \frac{\rho}{2} \ge d_{\Delta_w}(*, \Psi_w(p,s)) + \frac{1}{4} - \frac{\rho}{2}.$$
But since $\Psi_w$ is graph $f$-tame, $s' \leq t'$, and $d_{\Delta_w}(*,\Psi_w(p,t')) \leq n + \frac{3}{4}$, we know that $d_{\Delta_w}(*, \Psi_w(p,s')) \leq f \left( n + \frac{3}{4} \right)$.  This and the above inequality imply that 
$$d_{\Delta_w}(*, \Psi_w(p,s)) \leq f \left( n + \frac{3}{4} \right) + \frac{\rho}{2} - \frac{1}{4} = g(n).$$

So $\Psi_w$ is $g$-tame for all such $w$.  Therefore, $g$ is an intrinsic tame filling function for $\langle A | R \rangle$.\\

\end{proof}

\begin{Th}\label{main} Given a finite presentation $\mathcal{P} = \langle A | R \rangle$ such that for all $a \in A$, $a$ is not equal to the identity, there is an intrinsic tame filling function for $\mathcal{P}$ that is equivalent to the intrinsic diameter function of $\mathcal{P}$.
\end{Th}

\begin{proof}
Define $f: \N[\frac{1}{4}] \to \N[\frac{1}{4}]$ by $$f(n) = \IDiam(\lceil 2n+1 \rceil) + n + M_\mathcal{P} + 1.$$  Note that $f$ is strictly increasing, since $\IDiam$ is an increasing function.  Note also that $f$ is equivalent to $\IDiam$.  We will show that $f$ is an intrinsic graph tame filling function, and then apply Lemma~\ref{graph tame}. 

Let $w =_G 1$ and let $\Delta_0$ be any van Kampen diagram for $w$.  Let $T_0$ be a tree of $\Delta_0$-geodesics out of $*$ and let $\Psi_{T_0}$ be a 1-combing of $(\Delta_0, \partial \Delta_0)$ based at $*$ respecting $T_0$.  We will construct a sequence $(\Delta_0, T_0, \Psi_{T_0}), (\Delta_1, T_1, \Psi_{T_1}), \dots, (\Delta_n, T_n, \Psi_{T_n})$ where each $\Delta_i$ is a van Kampen diagram for $w$, each $T_i$ is a spanning tree of $\Delta_i$-geodesics out of $*$, and each $\Psi_{T_i}$ is a 1-combing of $(\Delta_i, \partial \Delta_i)$ respecting $T_i$, and then show that $\Psi_{T_n}$ is graph $f$-tame.

Given that we have constructed $\Delta_i$, $T_i$, and $\Psi_{T_i}$, if $\Psi_{T_i}$ is graph $f$-tame, then $n = i$ and we are done.  Otherwise, construct $\Delta_{i+1}$, $T_{i+1}$, and $\Psi_{T_{i+1}}$ as follows.  We will start with the construction of $\Delta_{i+1}$ by finding 1-cells in $\Delta_i$ whose $T_i$-icicles contain points too far away from $*$ for $\Psi_{T_i}$ to be $f$-tame.  We will then replace the bodies of these icicles in $\Delta_i$ with diagrams that have smaller diameter using Proposition~\ref{fat diagrams}; the resulting diagram will be $\Delta_{i+1}$.

Recall the definition of the points prior to $x$ in a 1-combing $\Psi$ of the pair $(C,D)$: 
$$P(\Psi,x) = \{ y \in C | \text{there is } p \in D \text{ and } 0 \leq s \leq t \leq 1 \text{ with } \Psi(p,t) = x \text{ and } \Psi(p,s) = y \}.$$ 
For $x \in \Delta_i^{(1)}$, let $$M_i(x) = \text{max} \{d_{\Delta_i}(*, y) | y \in \Delta_i^{(1)} \cap P(\Psi_{T_i},x) \},$$ the largest distance from $*$ that occurs in the 1-skeleton prior to $x$ in $\Psi_{T_i}$.  Note that, since $\Psi_{T_i}$ is not graph $f$-tame, there exists $x \in \Delta_i^{(1)}$ such that $M_i(x) > f(d_{\Delta_i}(*,x))$.  Then let 
$$N_i = \text{max} \{ d_{\Delta_i}(*,x) | x \in \Delta_i^{(1)} \text{ and } M_i(x) > f(d_{\Delta_i}(*,x)) \}.$$
Note that this maximum exists, since $\Delta_i$ is a finite complex, and therefore distances in the complex are bounded.  Also note that if $x \in T_i$ and $\Psi_{T_i}(p,t) = x$, since $\Psi_{T_i}$ respects $T_i$, we have that for all $s \leq t$, $\Psi_{T_i}(p,s)$ is on the simple path in $T_i$ from $*$ to $x$.  Since this path is a $\Delta_i$-geodesic, every point prior to $x$ is at least as close to $*$ as $x$, so $M_i(x) = d_{\Delta_i}(*,x) < f(d_{\Delta_i}(*,x))$.  Hence, if $d_{\Delta_i}(*,x) = N_i$ and $M_i(x) > f(N_i)$, then $x \not \in T_i$.

Therefore, let 
$$E_i = \left\{ e \in E(\Delta_i \setminus T_i) | d_{\Delta_i}(*,e) = N_i \text{ and there exists } x \in e \text{ such that } M_i(x) > f(N_i)\right\}.$$
Define a partial order $\leq_i$ on $E_i$ by $e \leq_i e'$ if and only if $e$ is contained in the $T_i$-icicle at $e'$ (and therefore the $T_i$-icicle at $e$ is contained in the $T_i$-icicle at $e'$ by Lemma~\ref{Icicle relationships}).  Let $F_i \subseteq E_i$ be the set of maximal elements of $E_i$ with respect to $\leq_i$.  Let $F_i = \{e_1, \dots, e_{m_i}\}$.  It is the bodies of the $T_i$-icicles at these 1-cells that we will replace to construct $\Delta_{i+1}$.

For each $j = 1, \dots, m_i$, let $I_j$ be the $T_i$-icicle at $e_j$ and let $D_j$ be the body of $I_j$.  Let $\alpha_j$ be the tail of $I_j$, and let $*_j$ be the vertex at which the tail and body of $I_j$ meet.  Let $x_j$ and $y_j$ be the endpoints of $e_j$ and let $\vec{e_j}$ be the directed edge from $x_j$ to $y_j$ corresponding to $e_j$.  Let $\beta_{x_j}$ and $\beta_{y_j}$ be the unique simple paths in $T$ from $*_j$ to $x_j$ and $y_j$, respectively.  Let $\gamma_j = \beta_{x_j} \cdot \vec{e_j} \cdot \overline{\beta_{y_j}}$.  Note that $\gamma_j$ is a simple circuit that bounds $D_j$.  Let $w_j$ be the word labeling $\gamma_j$.  Then we may consider $D_j$ to be a van Kampen diagram for $w_j$ with basepoint $*_j$.

Since we know that $D_j$ is simply bounded, let $D_j'$ be a simply and geodesically bounded van Kampen diagram for $w_j$ such that $$\IDiam(D_j') \leq \text{max} \left(\IDiam(\ell(w_j)), \left\lfloor \frac{\ell(w_j)}{2} \right\rfloor + M_\mathcal{P}\right),$$ as is guaranteed to exist by Proposition~\ref{fat diagrams}.

Let $$\widehat{\Delta_i} = \Delta_i \setminus \bigcup_{j \in [m_i]} \interior{D_j}.$$  Note that since each $e \in F_i$ is maximal in $\leq_i$, for distinct 1-cells $e,e' \in F_i$, the interiors of the $T_i$-icicles at $e$ and $e'$ do not intersect by Lemma~\ref{Icicle relationships}.  So if $k \neq j$, $D_k \cap \interior{D_j} = \emptyset$ in $\Delta_i$.  Therefore, for each $j \in [m_i]$, $\partial D_j \subseteq \widehat{\Delta_i}$.

\begin{figure}[t]
    \includegraphics[width=.95\linewidth]{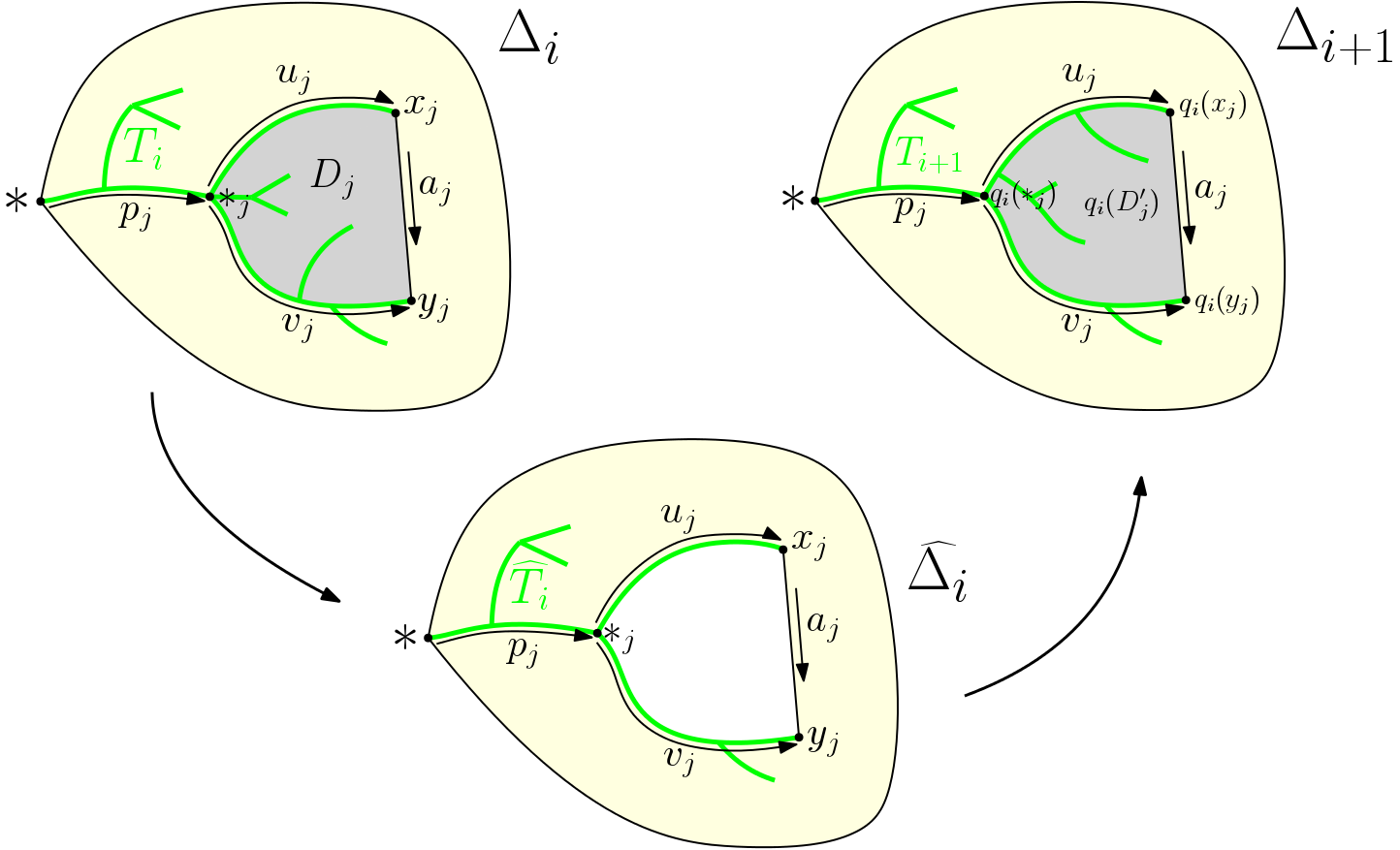}
    \caption{Constructing $\Delta_{i+1}$ and $T_{i+1}$}
    \label{fig:Delta_i+1}
\end{figure}

Construct $\Delta_{i+1}$ from $\widehat{\Delta_i}$ by gluing the basepoint of $D_j'$ to $*_j$ and then gluing the boundary of $D_j'$ to $\gamma_j$, gluing vertices to vertices and 1-cells to 1-cells with the same labels, for each $j \in [m_i]$, as depicted in Figure~\ref{fig:Delta_i+1}.  Let $$q_i: \widehat{\Delta_i} \sqcup \bigsqcup_{j \in [m_i]} D_j' \to \Delta_{i+1}$$ be the corresponding quotient map.  Note that, since each $D_j$ and $D_j'$ are simply bounded, $q_i$ does not identify 1-cells or vertices that were distinct in $\Delta_i$, or in any $D_j'$.  As a result, $\Delta_{i+1}$ is planar and simply connected.  Since we have only replaced a subset of the interior of $\Delta_i$, $\Delta_{i+1}$ is still a van Kampen diagram for $w$.

We now move on to the construction of $T_{i+1}$.  Let $\widehat{T_i} = T_i \cap \widehat{\Delta_i}$.  $T_{i+1}$ will be an extension of $q_i(\widehat{T_i})$.  I claim that $\widehat{T_i}$ is a spanning tree of $\widehat{\Delta_i}$.  Note that since $T_i$ is a tree and $\widehat{T_i} \subseteq T_i$, there are no simple circuits in $\widehat{T_i}$.  Also, since $T_i$ is a spanning tree of $\Delta_i$, $\widehat{T_i}$ contains every vertex in $\widehat{\Delta_i}$.  So we only need to show that $\widehat{T_i}$ is connected.  It suffices to show that, for every $p \in \widehat{\Delta_i}^{(0)}$, there is a path from $*$ to $p$ in $\widehat{T_i}$.  Let $\eta$ be the unique simple path from $*$ to $p$ in $T_i$.  Suppose that for some $j \in [m_i]$, $\eta$ intersects $D_j$.  Then let $q$ be the last vertex of $\eta$ that is in $D_j$.  Since $p \not \in \text{int}(D_j)$, we know that $q \in \partial D_j$.  Without loss of generality, assume that $q$ lies on the simple path from $*$ to $x_j$ in $T_i$.  Therefore, the segment of $\eta$ from $*$ to $q$ is contained in the path from $*$ to $x_j$ in $T_i$, which does not intersect $\text{int}(D_j)$.  Since $q$ is the last vertex of $\eta$ in $D_j$, this implies that $\eta$ does not intersect $\text{int}(D_j)$.  Therefore, $\eta \subseteq \widehat{T_i}$.  

\begin{figure}[h!]
    \includegraphics[width=.6\linewidth]{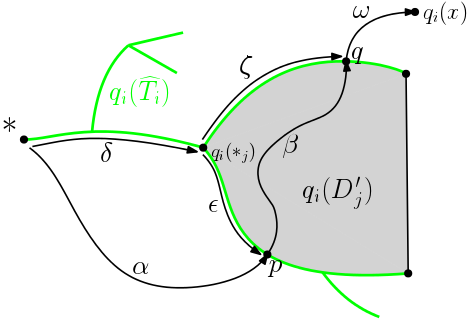}
    \caption{Constructing a $\Delta_{i+1}$-geodesic $\eta$ to stay inside $\widehat{\Delta_i}$ longer than $\gamma$.}
    \label{fig:eta_k+1}
\end{figure}

I also claim that $q_i(\widehat{T_i})$ is a tree of $\Delta_{i+1}$-geodesics out of $*$.  Since $\widehat{T_i}$ is a tree of $\widehat{\Delta_i}$-geodesics out of $*$, it is sufficient to show that for every $x \in \widehat{\Delta_i}^{(0)}$, we have that $d_{\widehat{\Delta_i}}(*,x) \leq d_{\Delta_{i+1}}(*,q_i(x))$.  Suppose by way of contradiction that there is some $x \in \widehat{\Delta_i}^{(0)}$ such that no $\Delta_{i+1}$-geodesic from $*$ to $q_i(x)$ is contained in $q_i(\widehat{\Delta_i})$.  Since $* \in q_i(\widehat{\Delta_i})$, each $\Delta_{i+1}$-geodesic from $*$ to $q_i(x)$ has some initial segment contained in $q_i(\widehat{\Delta_i})$.  Let $\gamma$ be a $\Delta_{i+1}$-geodesic from $*$ to $q_i(x)$ having the longest such initial segment.  Let $\alpha$ be this initial segment of $\gamma$, and let $p$ be its endpoint.  Let $q$ be the next vertex in $q_i(\widehat{\Delta_i})$ on $\gamma$ after $p$, let $\beta$ be the segment of $\gamma$ from $p$ to $q$, and let $\omega$ be the final segment of $\gamma$ from $q$ to $q_i(x)$.  Note that by assumption $p \neq q$, so $|\beta| \ge 1$ and all but the endpoints of $\beta$ are not contained in $q_i(\widehat{\Delta_i})$.  Hence, $\beta \subseteq q_i(D_j')$ for some $j \in [m]$ and $p,q \in q_i(\partial D_j')$.  Let $\delta$ be the simple path from $*$ to $q_i(*_j)$ in $q_i(\widehat{T_i})$ and let $\epsilon$ be the simple path from $q_i(*_j)$ to $p$ in $q_i(\widehat{T_i})$.  Finally, since $D_j'$ is geodesically bounded, there is a $q_i(D_j')$-geodesic $\zeta$ from $q_i(*_j)$ to $q$ in $q_i(\partial D_j')$.  See Figure~\ref{fig:eta_k+1}.  Now let $\eta = \delta \cdot \zeta \cdot \omega$. Note that $|\alpha| = |\delta \cdot \epsilon|$, since $\delta \cdot \epsilon$ is the simple path from $*$ to $p$ in $q_i(\widehat{T_i})$, and therefore a $q_i(\widehat{\Delta_i})$-geodesic from $*$ to $p$.  Also, $|\epsilon \cdot \beta| \ge |\zeta|$, since $\epsilon \cdot \beta$ is a path from $*_j$ to $q$ in $q_i(D_j')$.  Therefore, 
$$|\gamma| = |\alpha \cdot \beta \cdot \omega| = |\delta \cdot \epsilon \cdot \beta \cdot \omega| \ge |\delta \cdot \zeta \cdot \omega| = |\eta|.$$
Hence, $\eta$ is also a $\Delta_{i+1}$-geodesic from $*$ to $x$.  Now, the initial segment of $\eta$ contained in $q_i(\widehat{\Delta_i})$ contains $\delta \cdot \zeta$, a $\Delta_{i+1}$-geodesic from $*$ to $q$.  Since $\alpha \cdot \beta$ is also a $\Delta_{i+1}$-geodesic from $*$ to $q$ and $|\beta| \ge 1$, we have that $\delta \cdot \zeta$ is strictly longer than $\alpha$.  So $\eta$'s initial segment inside $q_i(\widehat{\Delta_i})$ is longer than that of $\gamma$.  This contradicts the way that $\gamma$ was chosen.  Therefore, for all $x \in \widehat{\Delta_i}^{(0)}$, there is a $\Delta_{i+1}$-geodesic from $*$ to $q_i(x)$ contained in $q_i(\widehat{\Delta_i})$.  So $d_{\widehat{\Delta_i}}(*,x) \leq d_{\Delta_{i+1}}(*,q_i(x))$ as desired.

\begin{figure}[h!]
    \includegraphics[width=.9\linewidth]{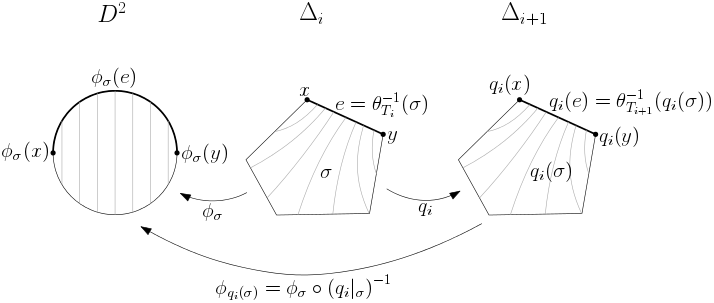}
    \caption{Choosing $\phi_{q_i(\sigma)}$ for a 2-cell $\sigma$ in $\widehat{\Delta_i}$.  The lighter curves depict some combing paths of $\Psi_{T_i}$ and their images.  As a result of this choice, the combing paths of $\Psi_{T_{i+1}}$ on $q_i(\sigma)$ will agree with those of $q_i \circ \Psi_{T_i}$.}
    \label{fig:Choosing Psi_T_i+1}
\end{figure}

Finally, let $T_{i+1}$ be an extension of $q_i(\widehat{T_i})$ to a geodesic spanning tree of $\Delta_{i+1}$.  Now we will choose $\Psi_{T_{i+1}}$ to be a particular 1-combing of $(\Delta_{i+1}, \partial \Delta_{i+1})$ based at $*$ respecting $T_{i+1}$, as guaranteed to exist by Proposition~\ref{Nice 1-combings exist}.  I claim that we may construct $\Psi_{T_{i+1}}$ such that, for all $x \in \partial \Delta_i$ and $t \in [0,1]$ with $\Psi_{T_{i+1}}(q_i(x),[t,1]) \subset q_i(\widehat{\Delta_i})$, there exists $t' \in [0,1]$ such that $\Psi_{T_{i+1}}(q_i(x),[t,1]) = q_i(\Psi_{T_i}(x,[t',1]))$. In other words, the images of the combing paths of $q_i \circ \Psi_{T_i}$ and $\Psi_{T_{i+1}}$ agree on end portions that stay in $q_i(\widehat{\Delta_i})$.  This will be useful for showing that $\Psi_{T_{i+1}}$ retains the progress that $\Psi_{T_i}$ has made towards achieving $f$-tameness.

To choose $\Psi_{T_{i+1}}$ in this way, recall that the images of individual combing paths of the 1-combings constructed in Section~\ref{1-combings section} are determined by the bijection $\theta_T$ between 1-cells outside of the spanning tree $T$ and 2-cells, as well as a choice of homeomorphism from each 2-cell $\sigma$ to $D^2$ that maps the edge $\theta_T^{-1}(\sigma)$ to $S^1_+$.  Since we chose $T_{i+1}$ to agree with $q_i(T_i)$ on $\widehat{\Delta_i}$, for any 1-cell $e \in E(\widehat{\Delta_i} \setminus T_i)$, we have that $q_i(e) \not \in T_{i+1}$ and $q_i(\partial I_e) = \partial I_{q_i(e)}$.  Therefore, for any 2-cell $\sigma \subseteq \widehat{\Delta_i}$, if $\sigma \subseteq I_e$, then $q_i(\sigma) \subseteq I_{q_i(e)}$.  Hence, $$q_i \circ \theta_{T_i}|_{E(\widehat{\Delta_i} \setminus T_i) \setminus F_i} = \theta_{T_{i+1}} \circ q_i|_{E(\widehat{\Delta_i} \setminus T_i) \setminus F_i}.$$  As a result, for each 2-cell $\sigma \subseteq \widehat{\Delta_i}$, if $\phi_\sigma: \sigma \to D^2$ is the homeomorphism chosen for $\sigma$ in the construction of $\Psi_{T_i}$, we may choose $\phi_{q_i(\sigma)} = \phi_\sigma \circ (q_i|_{\sigma})^{-1}$ as the homeomorphism from $q_i(\sigma)$ to $D^2$ in the construction of $\Psi_{T_{i+1}}$.  Figure~\ref{fig:Choosing Psi_T_i+1} represents this choice of $\phi_{q_i(\sigma)}$ pictorially.  Having chosen these homeomorphisms, the images of the combing paths of $\Psi_{T_{i+1}}$ on $q_i(\widehat{\Delta_i})$ are determined, and we are free to make any other choices necessary to finish the construction of $\Psi_{T_{i+1}}$ as in Section~\ref{1-combings section}.

Having chosen $\Psi_{T_{i+1}}$, let $x \in \partial \Delta_i$ and $t \in [0,1]$ with $\Psi_{T_{i+1}}(q_i(x),[t,1]) \subset q_i(\widehat{\Delta_i})$.  We need to show that there exists $t' \in [0,1]$ such that $\Psi_{T_{i+1}}(q_i(x),[t,1]) = q_i(\Psi_{T_i}(x,[t',1]))$.  

By Lemma~\ref{paths improved}, we know there exists a sequence $1 \ge t_1 > \dots > t_{m+1} > t_{m+2} = 0$ and a sequence of 1-cells $e_1, e_2, \dots, e_m$ in $\Delta_{i+1} \setminus T_{i+1}$ such that, if we let $x_j = \Psi_{T_{i+1}}(q_i(x),t_j)$ and $\sigma_j = \theta_{T_{i+1}}(e_j)$, we have that 

\begin{itemize}
\item $\Psi_{T_{i+1}}(q_i(x),[0, t_{m+1}])$ is the simple path from $*$ to $x_m$ in $T_{i+1}$,
\item $\Psi_{T_{i+1}}(q_i(x),[t_1,1]) = \{q_i(x)\}$, and
\item for all $j \in [m]$,

\begin{itemize}
\item $x_j \in \interior{e_j}$,
\item $t_j = \text{min}\{s \in [0,1] | \Psi_{T_{i+1}}(q_i(x),s) = x_j \}$, and
\item $\Psi_{T_{i+1}}(q_i(x),[t_{j+1},t_j]) = \phi_{\sigma_j}^{-1}(h(\phi_{\sigma_j}(x_j),[0,1]))$.
\end{itemize}

\end{itemize} 

Similarly, there is a sequence $1 \ge t_1' > \dots > t_{l+1}' > t_{l+2}' = 0$ and a sequence of 1-cells $e_1', e_2', \dots, e_l'$ in $\Delta_i \setminus T_i$ such that, if we let $x_j' = \Psi_{T_i}(x,t_j')$ and $\sigma_j' = \theta_{T_i}(e_j')$, we have that 

\vbox{%
\begin{itemize}
\item $\Psi_{T_{i}}(x,[0, t_{m+1}'])$ is the simple path from $*$ to $x_m'$ in $T_{i}$,
\item $\Psi_{T_{i}}(x,[t_1',1]) = \{x\}$, and
\item for all $j \in [l]$,
\begin{itemize}
\item $x_j' \in \interior{e_j'}$,
\item $t_j' = \text{min}\{s \in [0,1] | \Psi_{T_{i}}(x,s) = x_j' \}$, and
\item $\Psi_{T_{i}}(x,[t_{j+1}',t_j']) = \phi_{\sigma_j'}^{-1}(h(\phi_{\sigma_j'}(x_j'),[0,1]))$.
\end{itemize}
\end{itemize}}

\begin{figure}[h!]
    \includegraphics[width=.9\linewidth]{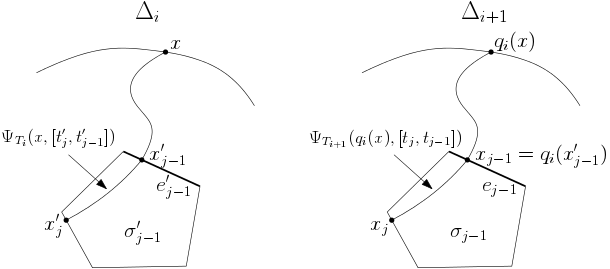}
    \caption{The situation at step $j$ of the induction, demonstrating that $\Psi_{T_{i+1}}(q_i(x),[t_j,1]) = q_i(\Psi_{T_i}(x,[t_j',1])).$}
    \label{fig:induction setup}
\end{figure}

If $t_1 \leq t \leq 1$, then we may simply let $t' = 1$.  Then $$\Psi_{T_{i+1}}(q_i(x),[t,1]) = \{q_i(x)\} = q_i(\Psi_{T_i}(x,[t',1])).$$  Otherwise, there is some $k \in [m+1]$ such that $t_{k+1} \leq t < t_k$.  For $j \in [k]$, suppose by induction that $\Psi_{T_{i+1}}(q_i(x),[t_{j-1},1]) = q_i(\Psi_{T_i}(x,[t_{j-1}',1]))$ and that $x_{j-1} = q_i(x_{j-1}')$.  This setup is shown in Figure~\ref{fig:induction setup}.  Then $e_{j-1} = q_i(e_{j-1}')$.  Since $\theta_{T_{i+1}}(e_{j-1}) = \sigma_{j-1}$ and $t_{j-1}$ is the smallest time with $\Psi_{T_{i+1}}(q_i(x),t_{j-1}) = x_{j-1}$, we must have that $\interior{\sigma_{j-1}} \cap \Psi_{T_{i+1}}(q_i(x),[t_j,t_{j-1}]) \neq \emptyset$.  We know that $\Psi_{T_{i+1}}(q_i(x),[t_j,t_{j-1}]) \subseteq q_i(\widehat{\Delta_i})$, so we must have $\sigma_{j-1} \subseteq q_i(\widehat{\Delta_i})$.  Hence, $$\sigma_{j-1} = \theta_{T_{i+1}}(e_{j-1}) = \theta_{T_{i+1}}(q_i(e_{j-1}')) = q_i(\theta_{T_i}(e_{j-1}')) = q_i(\sigma_{j-1}').$$  Then since we chose $\phi_{\sigma_{j-1}} = \phi_{\sigma_{j-1}'} \circ (q_i|_{\sigma_{j-1}'})^{-1}$ in the construction of $\Psi_{T_{i+1}}$,

\begin{align*}
\Psi_{T_{i+1}}(q_i(x),[t_j,t_{j-1}]) &= \phi_{\sigma_{j-1}}^{-1}( h( \phi_{\sigma_{j-1}}(x_{j-1}), [0,1] ) ) \\
&= q_i( \phi_{\sigma_{j-1}'}^{-1}( h( \phi_{\sigma_{j-1}'}(x_{j-1}'), [0,1] ) ) ) \\
&= q_i( \Psi_{T_i}(x, [t_j',t_{j-1}'])).
\end{align*}

Combining this result with the induction hypothesis gives that $$\Psi_{T_{i+1}}(q_i(x),[t_j,1]) = q_i(\Psi_{T_i}(x,[t_j',1])).$$  Also, 

\begin{align*}
x_j = \Psi_{T_{i+1}}(q_i(x),t_j) &= \phi_{\sigma_{j-1}}^{-1}( h( \phi_{\sigma_{j-1}}(x_{j-1}), 1 ) ) \\
&= q_i( \phi_{\sigma_{j-1}'}^{-1}( h( \phi_{\sigma_{j-1}'}(x_{j-1}'), 1 ) ) ) \\
&= q_i( \Psi_{T_i}(x, t_j')) = q_i(x_j').
\end{align*}

This completes the induction, establishing that $\Psi_{T_{i+1}}(q_i(x),[t_k,1]) = q_i(\Psi_{T_i}(x,[t_k',1]))$ and that $x_k = q_i(x_k')$.

Suppose $k \leq m$. Since $x_k = q_i(x_k')$, we again know that $e_k = q_i(e_k')$. Now since $t_k = \text{min}\{s \in [0,1] | \Psi_{T_{i+1}}(q_i(x),s) = x_k \}$ and $t < t_k$, we know that $\Psi_{T_{i+1}}(q_i(x),t) \neq x_k$.  Hence, $\Psi_{T_{i+1}}(q_i(x),[t,t_k])$ contains a point in $\interior{\sigma_k}$.  Just as above, this implies that $\sigma_k = q_i(\sigma_k')$ and therefore $\phi_{\sigma_k} = \phi_{\sigma_k'} \circ (q_i|_{\sigma_k'})^{-1}$.  Hence, 

\begin{align*}
\Psi_{T_{i+1}}(q_i(x),[t,t_k]) &= \phi_{\sigma_k}^{-1}( h( \phi_{\sigma_k}(x_{k-1}), [0,s] ) )  \text{, for some } s \in (0,1],\\
&= q_i( \phi_{\sigma_k'}^{-1}( h( \phi_{\sigma_k'}(x_{k-1}'), [0,s] ) ) ) \\
&= q_i( \Psi_{T_i}(x, [t',t_k'])),
\end{align*}
where $t' \in [t_{k+1}', t_k')$ such that $\phi_{\sigma_k'}( \Psi_{T_i}(x,t') ) = h( \phi_{\sigma_k'}(x_{k-1}'), s) = \phi_{\sigma_k}( \Psi_{T_{i+1}}(q_i(x),t) )$.

If instead $k = m+1$, then $\Psi_{T_{i+1}}(q_i(x),[0,t_k])$ is the unique simple path from $x_{m+2} = *$ to $x_{m+1}$ in $T_{i+1}$, and $y := \Psi_{T_{i+1}}(q_i(x),t)$ is a point along that path.  Since $x_{m+1} \in q_i(\widehat{\Delta_i})$ and we chose $T_{i+1} \subseteq q_i(\widehat{T_i})$, we have that $\Psi_{T_{i+1}}(q_i(x),[0,t_k]) \subseteq q_i(\widehat{\Delta_i})$.  Let $y' = (q_i|_{\widehat{\Delta_i}})^{-1}(y)$.  Again since $T_{i+1} \subseteq q_i(\widehat{T_i})$, we have that $y'$ is on that unique simple path from $*$ to $x_{m+1}'$, so $y' \in \Psi_{T_i}(x,[0,t_k'])$.  So let $t' \in [0,t_k']$ such that $\Psi_{T_i}(x, t') = y'$.  Then $\Psi_{T_{i+1}}(q_i(x),[t,t_k]) = q_i( \Psi_{T_{i+1}}(x,[t',t_k']) )$.

So in either case, there is some $t' \in [t_{k+1}',t_k']$ such that $$\Psi_{T_{i+1}}(q_i(x),[t,t_k]) = q_i( \Psi_{T_{i+1}}(x,[t',t_k']) ).$$  Therefore, 

\begin{align*}
\Psi_{T_{i+1}}(q_i(x),[t,1]) &= \Psi_{T_{i+1}}(q_i(x),[t,t_k]) \cup \Psi_{T_{i+1}}(q_i(x),[t_k,1]) \\
&= q_i( \Psi_{T_i}(x,[t',t_k']) ) \cup q_i( \Psi_{T_i}(x,[t_k',1]) ) \\
&= q_i( \Psi_{T_i}(x, [t',1]) ).
\end{align*}

So $\Psi_{T_{i+1}}$ has the desired property.

Our aim will now be to show that, if $\Psi_{T_{i+1}}$ is not graph $f$-tame, then $N_{i+1} < N_i$.  Recall that $N_i$ is the greatest distance from the basepoint of any $x \in \Delta_i^{(1)}$ such that $x$ demonstrates that $\Psi_{T_i}$ is not $f$-tame, i.e., such that there is a point $y$ prior to $x$ in $\Psi_{T_i}$ with $d_{\Delta_i}(*,y) > f(d_{\Delta_i}(*,x))$.  If the sequence $N$ is strictly decreasing, then at some point there will be no such $x$, and we will be left with a graph $f$-tame 1-combing.

Recall that each $e_j$ is a 1-cell in $\Delta_i$ such that points in the interior of $e_j$ are a distance $N_i$ away from $*$ and there is a point prior to $e_j$---and therefore, in the icicle at $e_j$---of distance greater than $f(N_i)$ away from $*$.  I claim that these obstructions to $f$-tameness have been removed in $\Delta_{i+1}$: that for each $j \in [m_i]$, every $y \in \Delta_{i+1}^{(1)}$ in the $T_{i+1}$-icicle at $q_i(e_j)$ has $d_{\Delta_{i+1}}(*,y) \leq f(N_i)$.

For $j \in [m_i]$, we know that 
$$\IDiam(D_j') \leq \text{max} \left( \IDiam(\ell(w_j)), \left\lfloor \frac{\ell(w_j)}{2} \right\rfloor + M_\mathcal{P} \right).$$
Now recall that $\alpha_j$ is the tail of the icicle at $e_j$ and that $\gamma_j = \beta_{x_j} \cdot \vec{e_j} \cdot \overline{\beta_{y_j}}$ bounds the body of that icicle.  Also recall that $N_i \in \N\left[\frac{1}{2}\right]$.  Since $d_{\Delta_{i+1}}(*,q_i(e_j)) = N_i$ and $q_i(\alpha_j \cdot \beta_{x_j})$ and $q_i(\alpha_j \cdot \beta_{y_j})$ are $\Delta_{i+1}$-geodesics, $|\alpha_j \cdot \gamma_j \cdot \overline{\alpha_j}| = |\alpha_j \cdot \beta_{x_j} \cdot \vec{e_j} \cdot \overline{\beta_{y_j}} \cdot \overline{\alpha_j}| \leq 2N_i + 1$.  So $\ell(w_j) = |\gamma_j| \leq 2N_i + 1 - 2|\alpha_j|$.  Therefore, 
\begin{align*}
\text{max} & \left(\IDiam(\ell(w_j)), \left\lfloor \frac{\ell(w_j)}{2} \right\rfloor + M_\mathcal{P} \right) \\
&\leq \IDiam(\ell(w_j)) + \left\lfloor \frac{\ell(w_j)}{2} \right\rfloor + M_\mathcal{P} \\
&\leq \IDiam( 2N_i + 1 - 2|\alpha_j| ) + \left\lfloor \frac{2N_i + 1 - 2|\alpha_j|}{2} \right\rfloor + M_\mathcal{P} \\
&\leq \IDiam( 2N_i + 1 ) + N_i + \frac{1}{2} - |\alpha_j| + M_\mathcal{P}.
\end{align*}
Now let $y'$ be a vertex in the $T_{i+1}$-icicle at $q_i(e_j)$ in $\Delta_{i+1}$.  If $y'$ is on the simple path from $*$ to $q_i(*_j)$ in $T_{i+1}$, then $d_{\Delta_{i+1}}(*,y') \leq N_i$.  Otherwise, since $q_i(*_j)$ is on the simple path from $*$ to $y'$ in $T_{i+1}$\label{Icicle 3}, 
\begin{align*}
d_{\Delta_{i+1}}(*,y') &= d_{\Delta_{i+1}}(*,q_i(*_j)) + d_{\Delta_{i+1}}(q_i(*_j),y') \\
&\leq |\alpha_j| + \IDiam(D_j') \\
&\leq |\alpha_j| + \IDiam( 2N_i + 1 ) + N_i + \frac{1}{2} - |\alpha_j| + M_\mathcal{P} \\
&= \IDiam( 2N_i + 1 ) + N_i + \frac{1}{2} + M_\mathcal{P} = f(N_i) - \frac{1}{2}.
\end{align*}
Every point on a 1-cell in the icicle has coarse distance differing by $\frac{1}{2}$ from some vertex in the icicle, so for all $y \in \Delta_{i+1}^{(1)}$ in the $T_{i+1}$-icicle at $q_i(e_j)$, 
\begin{equation}\label{Icicle diameter inequality}
d_{\Delta_{i+1}}(*,y) \leq \text{max}(N_i + \frac{1}{2}, f(N_i)) = f(N_i)
\end{equation}

Now if $\Psi_{T_{i+1}}$ is not graph $f$-tame, then there exists some $x \in \Delta_{i+1}^{(1)}$ such that $M_{i+1}(x) > f(d_{\Delta_{i+1}}(*,x))$.  Then let $p \in \partial \Delta_{i+1}$ and $s,t \in [0,1]$ such that $s < t$, $\Psi_{T_{i+1}}(p,t) = x$, $y := \Psi_{T_{i+1}}(p,s) \in \Delta_{i+1}^{(1)}$, and $d_{\Delta_{i+1}}(*,y) = M_{i+1}(x)$.  I claim that $d_{\Delta_{i+1}}(*,x) < N_i$, which will be sufficient to show that $N_{i+1} < N_i$.  We will prove this by considering two cases.

Case 1: $y$ is in the $T_{i+1}$-icicle at $q_i(e_j)$ for some $j \in [m_i]$.  Then by inequality ~\ref{Icicle diameter inequality}, $d_{\Delta_{i+1}}(*,y) \leq f(N_i)$.  Since we chose $x$ and $y$ such that $f(d_{\Delta_{i+1}}(*,x)) < M_{i+1}(x) = d_{\Delta_{i+1}}(*,y)$, this shows that $f(d_{\Delta_{i+1}}(*,x)) < f(N_i)$.  Since $f$ is a strictly increasing function, this implies that $d_{\Delta_{i+1}}(*,x) < N_i$.

Case 2: $y$ is not in a $T_{i+1}$-icicle at $q_i(e_j)$ for any $j \in [m_i]$.  Then since $\Psi_{T_{i+1}}$ respects $T_{i+1}$, we know that for all $r \in [s, 1]$, $\Psi_{T_{i+1}}(p,r)$ is also not in the $T_{i+1}$-icicle at $q_i(e_j)$ for any $j \in [m_i]$.  So $\Psi_{T_{i+1}}(p,[s,1]) \subseteq q_i(\widehat{\Delta_i})$.  Based on the way that we chose $\Psi_{T_{i+1}}$ such that its combing paths agree with those of $q_i \circ \Psi_{T_i}$ on $q_i(\widehat{\Delta_i})$, we know that there exist $p' \in \partial \Delta_i$ and $s' \in [0,1]$ such that $q_i(\Psi_{T_i}(p',[s',1])) = \Psi_{T_{i+1}}(p,[s,1])$.  In particular, $q_i(p') = p$, $q_i(\Psi_{T_i}(p',s')) = y$, and there exists $t' \in [s',1]$ such that $q_i(\Psi_{T_i}(p',t')) = x$.  For ease of notation, let $y' = \Psi_{T_i}(p',s')$ and let $x' = \Psi_{T_i}(p',t')$.  Now we know that $d_{\Delta_i}(*,x') = d_{\Delta_{i+1}}(*,x)$ and $d_{\Delta_i}(*,y') = d_{\Delta_{i+1}}(*,y)$.  Since $y' \in P(\Psi_{T_i},x')$, we have that
$$M_i(x') \geq d_{\Delta_i}(*,y') = d_{\Delta_{i+1}}(*,y) = M_{i+1}(x) > f(d_{\Delta_{i+1}}(*,x)) = f(d_{\Delta_i}(*,x')).$$

This implies that $d_{\Delta_i}(*,x') \leq N_i$, by the definition of $N_i$.  But if $d_{\Delta_i}(*,x') = N_i$, then $x'$ would be on some 1-cell in $E_i$, and therefore in the $T_i$-icicle at $e_j$ for some $j \in [m_i]$.  Since $T_{i+1} \cap q_i(\widehat{\Delta_i}) = q_i(\widehat{T_i})$, this would imply that $x$ is in the $T_{i+1}$-icicle at $q_i(e_j)$, which is a contradiction.  Therefore, $d_{\Delta_{i+1}}(*,x) = d_{\Delta_i}(*,x') < N_i$.  This concludes case 2 and proves that $d_{\Delta_{i+1}}(*,x) < N_i$.

Since $N_{i+1}$ is the maximum of $d_{\Delta_{i+1}}(*,x)$ for all such $x$, this implies that $N_{i+1} < N_i$.  But for all $i$ such that $N_i$ is defined, $N_i > 0$.  Therefore, this sequence $(\Delta_i,T_i,\Psi_{T_i})$ must end.  Based on the way the sequence was constructed, this implies that for some $n \in \N$, $\Psi_{T_n}$ is graph $f$-tame.  Since $w$ was an arbitrary word with $w =_G 1$ and there exists a van Kampen diagram $\Delta_n$ for $w$ with a graph $f$-tame 1-combing, $f$ is an intrinsic graph tame filling function for $\langle A | R \rangle$.  Then by Lemma~\ref{graph tame}, $f$ is equivalent to an intrinsic tame filling function for $\mathcal{P}$.  Since this equivalence of functions is transitive, there is an intrinsic tame filling function for $\mathcal{P}$ that is equivalent to $\IDiam$.

\end{proof}

\begin{remark} Unfortunately, a similar proof strategy for the extrinsic version of this theorem would fall apart.  It is necessary to choose the spanning trees $T_i$ to be trees of $\Delta_i$-geodesics out of $*$.  The fact that the paths used are $\Delta_i$-geodesics is crucial to attain inequality ~\ref{Icicle diameter inequality}.  Reformulating this inequality for the extrinsic case would require paths that are geodesics in the Cayley graph, since extrinsic diameter is measured using distance in the Cayley graph.  Since it is known that the intrinsic diameter functions of some groups grow strictly faster then their extrinsic diameter functions, it is not possible in general to find van Kampen diagrams with spanning trees of Cayley graph geodesics.  In fact, in the extrinsic case, it is not even clear if it is possible in general to find a filling with an $f$-tame path to each vertex of each diagram, no matter how fast the function $f$ grows; in the intrinsic case, optimally tame paths to each vertex are simply handed to us in the form of geodesics in the diagram.
\end{remark}

\section{Other Possible Refinements of the Intrinsic Diameter Function}
Given that intrinsic tame filling functions do not provide a proper refinement of the invariant given by intrinsic diameter functions, we are left with the question of whether or not there is any way to refine this invariant by measuring something along the same lines as intrinsic tame filling functions.  This statement is intentionally vague, given that an intuitive sense of ``what intrinsic tame filling functions measure" is rather subjective.  In this section, we suggest two possible refinements motivated by different perspectives on what intrinsic tame filling functions measure.

One can consider a 1-combing to describe a particular homotopy between the boundary circuit of a van Kampen diagram and the basepoint.  Since van Kampen diagrams embed into the Cayley complex of a presentation, so does this homotopy.  Then a choice of a van Kampen diagram and 1-combing for every word equal to the identity in the group gives us a homotopy in the Cayley complex of every loop in the Cayley graph down to the identity vertex.  We can then think of an intrinsic tame filling function as a bound for how far away from the identity a point on a loop is allowed to travel after it has already been brought within a given distance of the identity by the chosen homotopy (where distance here is measured within the image of the homotopy, not within the Cayley complex itself).  In other words, if the homotopy brings a point within some distance of the identity---its ultimate destination---the intrinsic tame filling function limits how far the homotopy is allowed to push that point away from the identity from then on.
 
From this perspective, we may strengthen the notion of intrinsic tame filling function by placing more firm restrictions on what the homotopy is allowed to do to points that have gotten within a certain distance of the identity.  Below, we define a potential filling function by restricting the total change in ``elevation"---in other words, distance from the basepoint---such a point is allowed to accumulate before it reaches the identity, rather than just restricting its maximum distance from the identity.

\begin{Def} Given a van Kampen diagram $\Delta$, a $1$-combing $\Psi$ of $(\Delta, \partial \Delta)$, $x \in \partial \Delta$, and $t \in [0,1]$, let $V(\Psi,x,t)$ be the total variation of the function $d := d_\Delta(*,\Psi(x,\cdot)): [0,1] \to \N[\frac{1}{4}]$ on the interval $[0,t]$.  To be precise, if there is a finite sequence $0 = t_1 < t_2 < \dots < t_k = t$ such that for all $i \in [k-1]$ we have that $d(t_{i+1}) \neq d(t_i)$ and for all $s \in (t_i, t_{i+1})$ either $d([t_i,s]) = \{d(t_i)\}$ or $d([s,t_{i+1}]) = \{d(t_{i+1})\}$, then 
$$V(\Psi,x,t) = \sum_{i=1}^{k-1} |d(t_{i+1}) - d(t_i)|.$$
Otherwise, $V(\Psi,x,t) = \infty$.
\end{Def}

\begin{Def} Let $f: \N[\frac{1}{4}] \to \N[\frac{1}{4}]$ be a function.  A 1-combing $\Psi$ of $(\Delta, \partial \Delta)$ is \emph{variation $f$-tame} if for all $n \in \N[\frac{1}{4}]$ and for all $x \in \partial \Delta$ and $t \in [0,1]$, if $d_{\Delta}(*,\Psi(x,t)) \leq n$, then $V(\Psi,x,t) \leq f(n)$.
\end{Def}

\begin{Def} Let $f: \N[\frac{1}{4}] \to \N[\frac{1}{4}]$ be non-decreasing.  $f$ is an \emph{intrinsic variation-tame filling function} for $\langle A | R \rangle$ if, for all $w \in (A \cup A^{-1})^*$ with $w =_G 1$, there is a van Kampen diagram $\Delta_w$ for $w$ and a 1-combing $\Psi_w$ of $\Delta_w$ such that $\Psi_w$ is variation $f$-tame.
\end{Def}

Note that for any $\Delta$, $\Psi$, $x$, and $t$ as in the definition of $V(\Psi,x,t)$, we have that for any $s \leq t$, $d_\Delta(*,\Psi(x,s)) \leq V(\Psi,x,t)$, and therefore any intrinsic variation-tame filling function is an intrinsic tame filling function that uses the same choice of $\Delta_w$ and $\Psi_w$ for each $w \in (A \cup A^{-1})^*$ with $w =_G 1$.\\

The second possible refinement comes from a completely different perspective on what intrinsic tame filling functions measure.  The proof of Theorem~\ref{main} essentially comes down to two facts:
\begin{enumerate}
\item For any geodesic spanning tree of a van Kampen diagram, there is a 1-combing that is $f$-tame where $f(n) = \text{max}\{\IDiam(I_e) | d_\Delta(*,e) \leq n\}$.  In a vague sense, the 1-combing is as tame as the diameter of the icicles of the tree.
\item It is possible to replace the icicles with versions of themselves that have almost minimal diameter.
\end{enumerate}
From a certain perspective, these facts imply that the 1-combing portion of the definition of intrinsic tame filling function is superfluous; there is a particular type of 1-combing that we can restrict ourselves to and acquire intrinsic tame filling functions that grow as slowly as possible, and their tameness can be measured by the diameter of the icicles of the tree without talking about 1-combings at all.  In other words, defining intrinsic tame filling functions based only on diameters of icicles would result in exactly the same group invariant.

From this perspective, the proof of (2) works because icicles of the same tree intersect each other so nicely.  For any two icicles, either one is inside the other or they do not intersect (besides at their boundary).  This makes it easier to find a good order in which to replace icicles without having to worry about intersections messing up an area that has already been fixed; just replace bigger icicles first.  This way, there is no possibility that you ruin the nice diameter of a previously-replaced icicle by replacing the next icicle; if you later replace any part of a previously-replaced icicle, you must be replacing some icicle entirely contained within it, and this can never increase the diameter of the previously-replaced icicle.

This gives us a second way to try to strengthen the notion of intrinsic tame filling functions.  We could define new types of filling functions based on the diameters of a set of subcomplexes that need not intersect each other nicely, with the intention of breaking the proof of (2).  We have defined such a type of filling function below, simply using the set of all subcomplexes that are themselves van Kampen diagrams.

Since subcomplexes do not come equipped with a basepoint, it will first be convenient to define a type of diameter that does not take a basepoint into account.  Given a 2-complex $C$, the \emph{unbased intrinsic diameter} of $C$ is $$\overline{\IDiam}(C) = \text{max}\{d_{C}(x,y) | x,y \in C^{(0)}\}.$$

\begin{Def} Given a van Kampen diagram $\Delta$, a \emph{subdiagram} of $\Delta$ is a simply-connected subcomplex of $\Delta$.  Note that any subdiagram $D$ of a van Kampen diagram can itself be thought of as a van Kampen diagram without a chosen basepoint.  It will be useful to let $|\partial D|$ refer to the length of the boundary circuit of $D$ without specifying a basepoint.
\end{Def}

\begin{Def} Given a function $f: \N \to \N$ and a van Kampen diagram $\Delta$, $f$ is an \emph{intrinsic subdiagram diameter function} for $\Delta$ if, for every subdiagram $D$ of $\Delta$, $\overline{\IDiam}(D) \leq f(|\partial D|)$.
\end{Def}

\begin{Def} A non-decreasing function $f: \N \to \N$ is an \emph{intrinsic subdiagram diameter function} for $\langle A | R \rangle$ if, for all $w \in (A \cup A^{-1})^*$ with $w =_G 1$, there is a van Kampen diagram $\Delta_w$ for $w$ such that $f$ is a local intrinsic diameter function for $\Delta_w$.
\end{Def}

Note that any intrinsic subdiagram diameter function for a presentation is bounded below by the intrinsic diameter function of the presentation, since every van Kampen diagram is a subdiagram of itself.\\

It is left to future work to determine if either of these new potential filling functions are quasi-isometry invariants, and whether or not they can distinguish between groups with equivalent intrinsic diameter functions.  

\bigskip

\bibliographystyle{plain}
\bibliography{ITFF_equivalent_to_IDiam}

\bigskip
\bigskip

\small{
\textsc{Andrew Hayes,
Department of Mathematics,
University of Nebraska,
Lincoln, NE 68588-0130, USA}, \texttt{ahayes11@unl.edu}}

\end{document}